\documentclass[11pt,a4paper,reqno]{amsart}
\usepackage{amsmath}
\usepackage{amsfonts}
\usepackage[utf8]{inputenc}
\usepackage[english]{varioref}
\usepackage{dcolumn}
\usepackage[height=22cm , width = 16cm , top = 4cm , left = 3cm, a4paper]{geometry}
\usepackage[final]{graphicx}
\usepackage{epsfig}
\usepackage{pstricks}
\usepackage{psfrag}
\usepackage{booktabs}
\usepackage{delarray}
\usepackage{amssymb}
\usepackage{rotating}
\usepackage{subfigure}
\usepackage{layout}
\usepackage{amsthm}
\usepackage{amssymb,amsmath}

\usepackage{hyperref}
\usepackage{xcolor}
\usepackage{setspace}
\usepackage{colortbl}
\usepackage[format=hang]{caption}
\usepackage{caption}
\usepackage{epsfig}
\usepackage{epstopdf}
\usepackage{booktabs}

\usepackage{mathtools}
\usepackage{multirow}
\usepackage[noadjust]{cite}
\usepackage{xcolor}
\hypersetup{
    colorlinks,
    linkcolor={red!50!black},
    citecolor={blue!80!black}
}

\DeclareUnicodeCharacter{00A0}{~}
{
	{\newcommand{\pd}{\partial}
		{
			{

				\newcommand{\et}{\textit{et al.} }

				\newtheorem{thm}{Theorem}[section]

				\newtheorem{rem}[thm]{Remark}
				
				\numberwithin{equation}{section}
				\newtheorem{example}{Example}[section]
				\parindent 0pt

				\title[Semi-Analytical Methods for solving Population Balance models]{\textbf{Semi-Analytical Methods for Population Balance models involving Aggregation and Breakage processes: A comparative study }}
				
				\author[Shweta, Saddam Hussain and Kumar]{Shweta$^{\dag}$,  Saddam Hussain$^{\dag, *}$, and  Rajesh Kumar$^{\dag}$}
				\thanks{$^\dag$Department of Mathematics,
					Birla Institute of Technology and Science, Pilani, Rajasthan-333031, India.\\
					$^\ast$Corresponding author: Email: p20200438@pilani.bits-pilani.ac.in\\
						Shweta: Email: p20210073@pilani.bits-pilani.ac.in\\
				 Rajesh Kumar: Email: 	rajesh.kumar@pilani.bits-pilani.ac.in.}
				
\begin{document}
					\maketitle
					\begin{quote}
						{\textit{ Abstract:
								Population balance models often integrate fundamental kernels, including sum, gelling  and Brownian aggregation kernels. These kernels have demonstrated extensive utility across various disciplines such as aerosol physics,  chemical engineering,  astrophysics,  pharmaceutical sciences  and mathematical biology for the purpose of elucidating particle dynamics. The objective of this study is to refine the semi-analytical solutions derived from current methodologies in addressing the nonlinear aggregation and coupled aggregation-breakage population balance equation. This work presents a unique semi-analytical approach based on the homotopy analysis method (HAM) to solve  pure aggregation and couple aggregation-fragmentation population balance equations, which is an integro-partial differentia	equation. By decomposing the non-linear operator, we investigate how to utilize the convergence control parameter to expedite the convergence of the HAM solution towards its precise values in the proposed method.
}}
					\end{quote}
						\section{Introduction}
						
					The particulate process has significant impacts on Population Balance Equations (PBEs). It alters the physical features such as changes in size, composition, enthalpy and porosity of two or more particles. These processes are distinguished by continuous and dispersed phases, each containing particles with distinct characteristics. The particles may be droplets, crystals, or bubbles. Population balance models have been utilized in the past few decades to address diverse industrial-scale applications, including flame synthesis of materials and aerosols, crystallization, emulsions, polymerization/depolymerization,  sprayed fluidized bed granulation, and twin-screw granulation. Notable applications of these processes, which are simulated in real-world settings using PBEs, include protein filament division \cite{tournus2021insights}, milling processes \cite{wang2021multiscale}, fibrin clot formation \cite{rukhlenko2015mathematical} and magnetic nanoflower aggregation in biofluids \cite{neofytou2021simulation}. The characteristics exhibited by particles in the course of particulate processes are subject to the influence of various mechanisms, specifically growth, nucleation, aggregation and fragmentation. These mechanisms are systematically quantified and delineated through the utilization of number density functions. Our study seeks to elucidate the number density functions arising from particle coagulation (aggregation), fragmentation (breakage), and combined coagulation and fragmentation. Aggregation refers to the event of merging two or more particles to create a more extensive particle. The utilization of the aggregation process is evident in multiple engineering applications, such as colloidal processing, granulation, the flame synthesis of materials, and polymerization. Coalescence occurs when several particles, such as liquid droplets, combine to form a single particle. In contrast, fragmentation is an irreversible  phenomenon  wherein particles undergo disintegration into smaller-sized entities due to collisions or external forces.
					Notably, the aggregate  volume or mass remains invariant during these operations to ensure the constancy of the overall material quantity within the system. Smoluchowski \cite{smoluchowski1917mathematical} initially presented a discrete version of the coagulation equation to analyze the Brownian motion of particles.
					Over the years, considerable attention has been devoted by numerous researchers to both the continuous and  discrete mutations of Smoluchowski's equations (see \cite{xie2022solution, laurenccot2018uniqueness} and references cited therein). The equation in the continuous setting is as follows \cite{filbet2004numerical}
					\begin{align}\label{agg}
						\frac{\partial c(s,\tau)}{\partial \tau}=\frac{1}{2}\int_{0}^{s}w(s-\xi,\xi)c(s-\xi,\tau)c(\xi,\tau)d\xi-\int_{0}^{\infty}w(s,\xi)c(s,\tau)c(\xi,\tau)d\xi,\quad \tau \in[0,T], s,\xi \in(0,\infty)
					\end{align}
					with an initial condition $c(s,0)  = c_{0}(s)\geq 0.$ 
					The equation (1.1) refers to the temporal variation in concentration of particles $c(s,\tau)$ during the aggregation process. Concurrently, $w(s,\xi)$ signifies the rate at which particles with sizes $s$ and $\xi$ combine, resulting in the formation of particles with size $s+\xi$. The first component in the right-hand side of equation \eqref{agg} represents the generation of particles of size $s$, originating from the collision between particles of volume $s-\xi$ and $\xi$.
					The subsequent term signifies the cessation of particles with a size of $s$ from the system. The study is also extended to encompass a coupled aggregation-fragmentation  process and the mathematical model is presented as follows \cite{mccoy2003analytical}:
					\begin{align}\label{agg-brk}
						\frac{\partial c(s,\tau)}{\partial \tau}=&\frac{1}{2}\int_{0}^{s}w(s-\xi,\xi)c(s-\xi,\tau)c(\xi,\tau)d\xi-\int_{0}^{\infty}w(s,\xi)c(s,\tau)c(\xi,\tau)d\xi+\int_{s}^{\infty}\beta(s,\xi)\mathcal{S}(\xi)c(\xi,\tau)d\xi \nonumber
						\\&-\mathcal{S}(s)c(s,\tau).
					\end{align}
					The last two terms of the equation describe the process of breakage. The term $\beta(s,\xi)$ in the equation represents the breakage kernel, which explains the creation of particle $s$ from particle $\xi$. On the other hand, the selection rate at which a particle of size $s$ is chosen to break is denoted by $S(s)$. The third term in the equation indicates the formation of a particle of size $s$, while the last term gives the disappearance of particle $s$.
					\par In addition to the number density $c(s,\tau)$, certain integral properties, including moments, gain attraction due to their physical representation. The $j^{th}$ moment corresponding to the number density  distribution is specified as 
					\begin{align}\label{int}
						n_{j}(\tau) = \int_{0}^{\infty} s^{j} c(s,\tau) ds, \quad j=1,2,3,\cdots.
					\end{align}
					The zeroth order moment $n_{0}(\tau)$ represents the total number of particles in the system at time $\tau$. The first order moment $n_{1}$ corresponds to the system's total mass and the second order moment $n_{2}$ provides the energy dissipation in the system. 
					\subsection{State-of-the Art and Contribution}
					Many researchers have extensively explored these models, delving into both theoretical and numerical aspects, driven by their intricate structure and diverse applications across various fields. Notably, in recent years, high-speed computing has significantly spurred the development of numerical, analytical, and semi-analytical approaches for solving PBEs in various forms.  Due to the complex nature of these models and the lack of analytical solutions except for a few simple cases, several semi-analytical and numerical approaches are used to solve these problems approximately.
					\par Numerical methods for solving the coupled breakage-aggregation model \eqref{agg-brk} include the quadrature method of moments \cite{marchisio2003quadrature}, finite element method \cite{bie2018coupling}, finite volume scheme \cite{singh2019new, filbet2004mass, kumar2014convergence}, fixed pivot technique \cite{giri2013convergence}, fast Fourier transformation method \cite{ahrens2018fft}, and references therein. Furthermore, in  \cite{lee2000simultaneous}, authors addressed the coagulation accompanied by simultaneous binary breaking by applying a stochastic approach. Mahoney et al. applied the finite element approach for nucleation, growth, and aggregation equations in \cite{mahoney2002efficient}. Ranjbar et al. \cite{ranjbar2010numerical} used radial basis functions and Taylor's polynomials to solve \eqref{agg} numerically. All of these numerical techniques exhibit limitations owing to their reliance on non-physical assumptions, including discretization, linearization, and sets of basis functions. As the problem is nonlinear, readers are directed to \cite{singh2022challenges} for an overview of the difficulties of using these numerical methods in solving PBEs.
					\par 
					To address these limitations, many authors have expressed a keen interest in solving the ordinary, partial, and integro-partial differential equations through the application of semi-analytical techniques (SATs), see \cite{temimi2015computational, hussain2023semi, temimi2011approximate, temimi2016time} and references therein. Recently, several authors have developed interest in SATs to study PBEs that include Adomian decomposition method (ADM), Homotopy analysis method (HAM), Homotopy
					perturbation method (HPM). Singh \et \cite{singh2015adomian} explored ADM for constant, sum, and product coagulation kernels, i.e,  $w(s,\xi)=1, s+\xi, s\xi$ with initial guess $c_{0}=e^{-s}$, $c_{0}=\frac{e^{-s}}{s}$ and established the convergence analysis for the constant kernel.
					Kaur \et \cite{kaur2019analytical} addressed the convergence of a series solution calculated via HPM for pure aggregation equation. Further, authors in \cite{kaur2022approximate} discussed the HAM for the pure breakage and pure coagulation equations but no improvement is observed in the solution as authors take convergence control parameter $h=-1$, giving the same results as ADM \cite{singh2015adomian}.
					Kaushik and Kumar \cite{kaushik2023novel} examined the pure aggregation equation using optimized decomposition method (ODM). Further, they used Laplace transform based ODM for solving the breakage and coagulation equations \cite{kaushik2023laplace}.
					\subsection{Motivation}
					Studying PBEs is prompted by their inherent persistence in natural and industrial contexts. These ubiquitous processes have wide-ranging implications, from the coalescence of droplets in cloud formation to the formation of aggregates in pharmaceutical manufacturing.
					 Identifying and modeling these processes play a pivotal role in optimizing the behavior of intricate systems. 
					The HAM, designed to address integral and differential equations, and has been successfully utilized for many dynamical problems \cite{liao1992proposed, liao1997kind, liao1999explicit}. Unlike classic non-perturbation techniques, HAM does not depend on a specific basis function or other small/large physical parameters. It provides a more direct approach that ensures quick convergence of iterative series solutions. The primary advantage of HAM is its versatility in terms of convergence rate and region \cite{liao2003beyond}. Therefore, the first goal of this work is to improve the HAM for non-linear equations. In summary, we chose the accelerated homotopy analysis method \cite{hussain2023semi} because it is adaptable, numerically efficient, and holds promise for exploring complex real-world systems.  AHAM provides a comprehensive set of auxiliary functions, auxiliary linear operator,  initial approximations, and convergence control parameters. This versatility drives us to use the method effectively and quickly to solve aforementioned equations. The  AHAM is introduced to solve the aggregation and coupled aggregation-breakage equations and conduct a comparative analysis with existing methods. The algorithm's convergence analysis and error estimates for the aggregation equation are also provided. Four distinct types of kernels are taken into consideration: constant, sum, product, and Brownian kernel. Our subsequent objective involves implementing the proposed methodology to address the coupled aggregation-breakage equation. Additionally, our objective is to perform a convergence analysis for a nonlinear coupled aggregation-breakage equation by employing an alternative approach incorporating contraction mapping in the Banach space.
					\par	The subsequent sections of this article are structured as follows: Section 2 provides a concise elucidation of the HAM and AHAM. Section 3 outlines the application of AHAM in the context of aggregation and
					coupled aggregation-breakage equations. This section also furnishes an comprehensive analysis of the proposed scheme's convergence concerning the issues mentioned earlier. In Section 4, the numerical implemention of the proposed scheme is expounded. Proceeding further, five examples of aggregation and two examples of coupled aggregation-fragmentation equations are considered, and the resultant findings are systematically compared with existing methods ADM/HPM/HAM as well as ODM/LODM against exact analytical outcomes. Lastly, Section 5 summarizes the observations and discussions.
		\section{ Methodology}
		This section delves into the basics of essential SATs used to solve non-linear differential equations. Take into account the general non-linear differential equation
			\begin{align}\label{1}
				\mathcal{D}[c'(s,\tau),c(s,\tau)]=0,
			\end{align}
		with initial condition
			\begin{align}\label{2}
		c(s,0)=c_{0}(s).
		\end{align}
	Here $c(s,\tau)$ is an unknown function and $c'(s,\tau)=\frac{\partial c}{\partial \tau}$.
	\subsection{HAM}
Let us consider the general differential equation \eqref{1}. Subsequently, in accordance with the methodology proposed by HAM \cite{liao2003beyond}, a homotopy is constructed as follows: 
			\begin{equation}\label{3}
				(1-r)L\left[\Psi(s,\tau ; r)-c_{0}(s,\tau)\right]=r h H \mathcal{D}\left(\Psi^{\prime}(s,\tau ; r), \Psi(s,\tau ; r)\right),
			\end{equation}
	where $r\in[0,1]$ is an embedding parameter, $h$ is a non-zero convergence control parameter, $H(t)\neq0$ is an auxiliary function, $L$ is an linear operator and $c_{0}(s,\tau)$ is an initial condition. It is evident from \eqref{3} that as $r$ varies from 0 to 1, $\Psi(s,\tau ; r)$ transitions  from initial estimate $c_{0}(s,\tau) = \Psi(s,\tau ; 0)$ to the exact solution $ c(s,\tau) = \Psi(s,\tau ; 1)$.
	The fundamental concept underlying the HAM is that the solution to equation \eqref{3} can be expressed in the form of a power series in the variable $ r $, as follows:
				\begin{equation}\label{4}
					\Psi(s,\tau ; r)=\sum_{k=0}^{\infty} r^{k}\mu_{k}(s,\tau).
				\end{equation}
	With the appropriate selection of auxiliary parameters, the series given by equation \eqref{4} converges at $r=1$,  providing
\begin{equation}\label{5}
		c(s,\tau)=\Psi(s,\tau ; 1) =\sum_{k=0}^{\infty}\mu_{k}(s,\tau),
	\end{equation}
which must correspond to the solutions of the original nonlinear equation, as proved by Liao \cite{liao2004homotopy}. To obtain the individual components, substitute \eqref{4} in the equation \eqref{3} and compare the coefficients of the corresponding powers of $r$. This results in the zeroth-order deformation equation as
$$  L\left[\mu_{0}-c_{0}\right]=0,$$
and the high order deformation equations
\begin{equation}\label{6}
	 L\left[\mu_{i+1}\right] =L\left[\mu_{i}\right] + hH(t){Q_i} ,
\end{equation}
 where $Q_i$ is the homotopy polynomials for $ i \geq 0$ and can be determined as 
$$ Q_i = \frac{1}{i!}\frac{\partial^i}{\partial r^i}\mathcal{D}\bigg( \sum_{j=0}^{\infty}r^j\mu_j^{\prime}(\tau), \sum_{j=0}^{\infty}r^j\mu_j(\tau)\bigg)\bigg|_{r=0} .$$
		\subsection{Accelerated Homotopy Analysis Method}
			Let us consider the differential equation \eqref{1} into the form	
			\begin{equation}\label{7}
			\frac{\partial c}{\partial \tau}+L[c(s,\tau)]+M[c(s,\tau)]=g(s,\tau) ,
			\end{equation}
		where $L $ and $ M$ represent linear and non-linear operators, respectively, and $ g(s,\tau)$ is a source function with  initial guess $c(s,0)=c_{0}(s)$.
			Now, following the HAM, we formulate a homotopy as follows
			\begin{equation}\label{8}
			(1-r)\left[\frac{\partial \Psi(s,\tau ; r)}{\partial \tau}-\frac{\partial c_{0}}{\partial \tau}\right]-r h H\bigg[\frac{\partial \Psi(s,\tau ; r)}{\partial \tau}+L[\Psi(s,\tau ; r)]+M[\Psi(s,\tau ; r)]-g(s,\tau)\bigg]=0.
		\end{equation}
Similar to the HAM, here as well, $ r \in [0,1]$ represents embedding parameter, $h\neq0$ is the convergence control parameter and as $r$ varies
	from 0 to 1, $\Psi(s,\tau ; r)$ changes from the initial estimate $c_{0}(s) = \Psi(s,\tau ; 0)$ to the
	exact solution $ c(s,\tau)=\Psi(s,\tau ; 1)$, of the above defined  problem \eqref{7} with $H=1$.
The main idea behind accelerating HAM is to deal with the non-linear parameter $M$.
Let 
	\begin{equation}\label{10}
	M [c(s,\tau)]=\sum_{k=0}^{\infty}H_{k}r^{k},
\end{equation}
where $H_{k}$ are accelerated He's polynomial \cite{jasrotia2022accelerated} defined as
	\begin{align}\label{11}
H_{k}(\mu_{1},\mu_{2},\ldots,\mu_{k})=M\bigg(\sum_{j=0}^{k}\mu_{j}\bigg)-\sum_{j=0}^{k-1}H_{j},
\end{align}
with $H_{0}=M(c_{0})$.
Now, using \eqref{4} and \eqref{10}, \eqref{8} reduces to
	\begin{align}\label{12}
	(1-r)\left[\frac{\partial }{\partial \tau}(\sum_{k=0}^{\infty}\mu_{k}r^k)-\frac{\partial \mu_{0}}{\partial \tau}\right]-rH h \bigg[\frac{\partial }{\partial \tau}(\sum_{k=0}^{\infty}\mu_{k}r^k)+L(\sum_{k=0}^{\infty}\mu_{k}r^k)+H_{k}r^{k}-g(s,\tau)\bigg]=0.
\end{align}
	By collecting and comparing the coefficients of the same powers of $r$, likewise $r^{0},r^{1},\ldots,$ we obtain
	\begin{align}\label{iter}
		r^{0} &: \frac{\partial \mu_{0}}{\partial\tau}-\frac{\partial c_{0}}{\partial\tau}=0, \nonumber \\
				r^{1} &: \frac{\partial \mu_{1}}{\partial\tau}-h\bigg[\frac{\partial \mu_{0}}{\partial\tau} +L\mu_{0}(s,\tau)+H_{0}-g(s,\tau)\bigg]=0,\nonumber \\
					r^{2} &: \frac{\partial \mu_{2}}{\partial\tau}-\frac{\partial \mu_{1}}{\partial\tau}-h\bigg[\frac{\partial \mu_{1}}{\partial\tau} +L\mu_{1}(s,\tau)+H_{1}\bigg]=0, \nonumber \\ 
					&\vdots \nonumber \\
					r^{k} &: \frac{\partial \mu_{k}}{\partial\tau} -\frac{\partial \mu_{k-1}}{\partial\tau}-h\bigg[\frac{\partial \mu_{k-1}}{\partial\tau} +L\mu_{k-1}(s,\tau)+H_{k-1}\bigg]=0. 
			\end{align}
		The components of the series solution $\sum_{k=0}^{\infty}\mu_{k}(s,\tau)$, which contain the convergence control parameter $h$, can be obtained from the above relation. The optimal value of $h$ will provide the best desired approximation to the non-linear problem \eqref{7}.
\begin{rem}
The convergence control parameter plays a crucial role in influencing both the convergence rate and the accuracy of the AHAM solutions. Hence, an optimal selection of this auxiliary parameter is required. Several strategies for obtaining the optimal value of h are available in the literature, such as those presented in Liao's works \cite{liao2003beyond, liao2004homotopy}. The discrete square residual technique is used throughout this study to determine the best appropriate value of $h$. The residual error is
	\begin{align}\label{14}
		Res(s,\tau)=	\frac{\partial \psi_{k}(s,\tau)}{\partial \tau}+L\psi_{k}(s,\tau)+M\psi_{k}(s,\tau)-g(s,\tau),
	\end{align}
where $$ \psi_{k}(s,\tau)= \sum_{i=0}^{k}\mu_{i}(s,\tau), $$ is the truncated series solution and the corresponding error function $E[h]$ which requires minimization, is given by:
\begin{align}\label{15}
E(h)= \frac{1}{K^2}\sum_{i=0}^{K}\sum_{i=0}^{K}Res^2(s_{i},\tau_{i}),
\end{align}
where $K$ represents an integer. It is observed that $E(h)$ consists an unknown parameter $h$. As $E(h)$ approaches zero, the solution for AHAM converges more quickly. To ascertain the optimal value of 
$h$, the "Minimize" command in the symbolic computation program "MATHEMATICA" is employed. 
\end{rem}
\begin{rem}
	The AHAM offers various options, including diverse initial approximations, auxiliary functions, auxiliary linear operators, and convergence control parameters. This flexibility empowers us to effectively and efficiently address non-linear problems. The proposed approach yields an efficient and rapidly converging solution, providing significantly greater accuracy compared to other techniques available in the literature.
\end{rem}
\section{Development of the method for Aggregation and coupled aggregation-fragmentation Equations  } 
In this section, the derivation of the pure aggregation and coupled aggregation-fragmentation problems. Furthermore, the study thoroughly investigates and discusses the convergence criteria and error bounds are provided for the proposed AHAM in solving both the models. The theoretical convergence of AHAM with maximum error bounds is performed in the Banach space via contraction mapping.
\subsection{AHAM Implementation for Pure Aggregation } 
 On comparing \eqref{agg} and \eqref{7}, we have $g(s,\tau)=0$, $L[c(s,\tau)]=0$ and $M[c(s,\tau)]$= -RHS of equation \eqref{agg}.
Then the homotopy is constructed as 
\begin{align*}
	(1-r)\left[\frac{\partial \Psi(s,\tau ; r)}{\partial \tau}-\frac{\partial c_{0}}{\partial \tau}\right]-&r h \bigg[\frac{\partial \Psi(s,\tau ; r)}{\partial \tau}-\frac{1}{2}\int_{0}^{s}w(s-\xi,\xi)c(s-\xi,\tau)c(\xi,\tau)d\xi\\&+\int_{0}^{\infty}w(s,\xi)c(s,\tau)c(\xi,\tau)d\xi\bigg]=0.
\end{align*}
In view of \eqref{11}, $H_{k}$'s for the non-linear aggregation equation \eqref{agg} are defined as
\begin{align}\label{16}
H_{k}(\mu_{1},\mu_{2},\dots,\mu_{k})=&-\frac{1}{2}\int_{0}^{s}w(s-\xi,\xi)\sum_{i=0}^{k}\mu_{i}(s-\xi,\tau)\sum_{i=0}^{k}\mu_{i}(\xi,\tau)d\xi \nonumber \\&+\int_{0}^{\infty}w(s,\xi)\sum_{i=0}^{k}\mu_{i}(s,\tau)\sum_{i=0}^{k}\mu_{i}(\xi,\tau)d\xi,
\end{align}
with $ H_{0}=M[c_{0}]$. By following the same procedure as we did in previous section, iterations to solve the above problem are
 	\begin{align}\label{iteragg}
  \mu_{0}=& c_{0}(s), \nonumber \\
  \mu_{1}=&\int_{0}^{\tau}h\bigg[\frac{\partial \mu_{0}}{\partial\rho} +H_{0}\bigg]{d\rho}, \nonumber \\
 	&\vdots \nonumber \\
   \mu_{i}=&\int_{0}^{\tau}\bigg(\frac{\partial \mu_{i-1}}{\partial\rho}+h\bigg[\frac{\partial \mu_{i-1}}{\partial\rho}+ H_{i-1}\bigg]\bigg)  {d\rho}, 2\leq i\leq k.
 \end{align}
These series solutions components include the convergence control parameter, which is determined using the discrete square residual error approach mentioned in Remark 2.1. Then the solution to the problem \eqref{agg} is presesnted as
	\begin{align*}
	c(s,\tau)=\lim_{r \to 1} \Psi(s,\tau,r) = \sum_{k=0}^{\infty}\mu_{k}(s,\tau),
\end{align*}
and for approximation, $k$-th order truncated series solution is provided  as
$\psi_{k}=\sum_{i=0}^{k}\mu_{i}(s,\tau).$
\subsection{Convergence Analysis for Pure Aggregation Equation}
To establish the convergence analysis of the proposed method, consider the Banach space $\mathbb{X}_{1}= (\mathcal{C}(\mathcal{L}^{1}[0 , \infty):[0 ,T])  , \|.\|)$  with the norm defined by
\begin{align}\label{NORM}
\|c\|	= \sup_{\tau\in [0,\tau_{0}]} \int_{0}^{\infty}\left|c(s,\tau)  \right|ds < D. 
\end{align}
Let us write equation \eqref{agg} into the operator form as
\begin{align}\label{17}
	c(s,\tau) = \mathcal{A}[c],
\end{align}
where $\mathcal{A}[c]$ is a non-linear operator defined by
\begin{align}\label{18}
	 \mathcal{A}[c]= c_{0}(s)+ \int_{0}^{\tau}A(c)d\rho.
\end{align}
where $$A(c)= \frac{1}{2}\int_{0}^{s}w(s-\xi,\xi)c(s-\xi,\rho)c(\xi,\rho)d\xi-\int_{0}^{\infty}w(s,\xi)c(s,\rho)c(\xi,\rho)d\xi$$
To demonstrate the contractive nature of the non-linear operator $A$, the following equivalent form of equation \eqref{18} is taken 
\begin{align*}
\frac{\partial }{\partial\tau}\bigg(c(s,\tau)\exp[G(s,\tau,c)]\bigg)=\frac{1}{2}\exp[G(s,\tau,c)]\int_{0}^{s}w(s-\xi,\xi)c(s-\xi,\tau)c(\xi,\tau)d\xi,
\end{align*}
where 
\begin{align}\label{D}
	G(s,\tau,c)=\int_{0}^{\tau}\int_{0}^{\infty}w(s,\xi)c(\xi,\rho)d\xi d\rho.
\end{align}
Thus, the above equation can be rewritten as $$ c=\widetilde{\mathcal{A}}c,$$ 
where
\begin{align}\label{19}
	\widetilde{\mathcal{A}}c=c_{0}(s)\exp[-G(s,\tau,c)]+\frac{1}{2}\int_{0}^{\tau}\exp[G(s,\rho,c)-G(s,\tau,c)]\int_{0}^{s}w(s-\xi,\xi)c(s-\xi,\rho)c(\xi,\rho)d\xi d\rho.
\end{align}
Since $\mathcal{A}$ and $\widetilde{\mathcal{A}}$ are equivalent, it is sufficient to show that $\widetilde{\mathcal{A}}$ is contractive.
	\begin{thm}\label{thm1}
	The operator $\widetilde{\mathcal{A}}$  is contractive on $\mathbb{X}_{1}$, i.e; $\|\widetilde{\mathcal{A}}c-\widetilde{\mathcal{A}}c^{\ast}\| \leq \gamma \|c-c^{\ast}\|$, $\forall$ $c,c^{\ast}\in\mathbb{X}_{1}$ with 
	\begin{itemize}
		\item[1] $w(s,\xi)=1$, for all $s,\xi\in(0,\infty)$,
		\item[2] $ \gamma= \widetilde{T} e^{\widetilde{T} D}( \|c_{0}\|+\frac{1}{2}\widetilde{T}D^2+\widetilde{T}D)< 1$ where $\widetilde{T}= min{\{ \tau_{0}, \tau_{1}\}}. $
	\end{itemize}
\end{thm}
	\begin{proof}
	First, we start this by showing $ \|\widetilde{\mathcal{A}}c\| < D$ for small $\widetilde{T} > 0$.
 Let us consider
	\begin{align*}
	\|\widetilde{\mathcal{A}}[c]\|\leq &	\|c_{0}\exp[-G(s,\tau,c)]	\|+\frac{1}{2}	\|\int_{0}^{\tau}\exp[G(s,\rho,c)-G(s,\tau,c)]\int_{0}^{s}c(s-\xi,\rho)c(\xi,\rho)d\xi d\rho	\| \\
		\leq &	\|c_{0}\|+\frac{1}{2}	\|\int_{0}^{\tau}\exp\bigg[-\int_{\rho}^{\tau}\int_{0}^{\infty}c(\xi,\tau) d\xi d\rho\bigg]\int_{0}^{s}c(s-\xi,\rho)c(\xi,\rho)d\xi d\rho \| \\
			\leq &	\|c_{0}\|+\frac{1}{2}	\int_{0}^{\infty}\int_{0}^{\tau}\int_{0}^{s}c(s-\xi,\rho) c(\xi,\rho) d\xi d\rho ds .	
	\end{align*}
Further, by using equation \eqref{NORM} and changing the order of integration, we get
		\begin{align*}
		\|\widetilde{\mathcal{A}}[c]\|\leq & \|c_{0}\|+\frac{1}{2}\int_{0}^{\tau}	\int_{0}^{\infty}\int_{0}^{\infty}c(z,\rho) c(\xi,\rho) dz d\xi d\rho \\
		  & \leq\|c_{0}\|+\frac{1}{2} D^2 \tau.	
	\end{align*}	
Now, 	$	\|\widetilde{\mathcal{A}}[c]\|< D $	 holds true if $\|c_{0}\|+\frac{1}{2} D^2 \tau_{0} \leq D $ for a suitable $\tau=\tau_{0}$. This inequality holds if $\tau_{0}  \leq \frac{1}{2\|c_{0}\|}$
and 
\begin{align*}
	 \frac{1-\sqrt{1-2\tau_{0}\|c_{0}\|}}{\tau_{0}}  \leq D \leq \frac{1+\sqrt{1-2\tau_{0}\|c_{0}\|}}{\tau_{0}}. 
	 	\end{align*}
	Now, we will emphasize showing that the mapping $\widetilde{A}$ is contractive. Consider the following:
	\begin{align*}
		\widetilde{\mathcal{A}}[c]-\widetilde{\mathcal{A}}[c^{\ast}]= & c_{0} \mathcal{U}(s,0,\tau)+\frac{1}{2}\int_{0}^{\tau}\mathcal{U}(s,\rho,\tau)\int_{0}^{s}c(s-\xi,\rho)c(\xi,\rho)d\xi d\rho+\frac{1}{2}\int_{0}^{\tau}\exp[G(s,\rho,c)-G(s,\tau,c^{\ast})]\\& \bigg[\int_{0}^{s} c^{\ast}(s-\xi,\rho)[c(\xi,\rho)-c^{\ast}(\xi,\rho)] d\xi+\int_{0}^{s} c(\xi,\rho)[c(s-\xi,\rho)-c^{\ast}(s-\xi,\rho)] d\xi \bigg] d\rho,
	\end{align*}
where $$\mathcal{U}(s,\rho,\tau) =\exp[G(s,\rho,c)-G(s,\tau,c)]-\exp[G(s,\rho,c^{\ast})-G(s,\tau,c^{\ast})]. $$ It can be easily demonstrated that
	\begin{align*}
	|\mathcal{U}(s,\rho,\tau)|\leq \delta \|c-c^{\ast}\|,
\end{align*}
where, $\delta = \tau\exp[\tau J] $ and $J = max{\{\|c\|,\|c^{\ast}\|\}}.$
Therefore, we have 
\begin{align*}
	\widetilde{\mathcal{A}}[c]-	\widetilde{\mathcal{A}}[c]^{\ast}\leq& \delta\|c_{0}\|\|c-c^{\ast}\|+\frac{1}{2}\delta\|c-c^{\ast}\|\int_{0}^{\tau}\|c\|^2 d\rho+\frac{1}{2} \int_{0}^{\tau}\delta[(\|c\|+\|c^{\ast}\|)\|c-c^{\ast}\|]d\rho \\ 
	\leq& \delta\bigg[ \|c_{0}\|+\frac{1}{2}\tau\|c\|^2+\frac{1}{2}\tau(\|c\|+\|c^{\ast}\|)\bigg]\|c-c^{\ast}\|
\end{align*}
if $ \gamma= \delta( \|c_{0}\|+\frac{1}{2}\tau_{1}D^2+\tau_{1}D)< 1$ for suitable $\tau_{1}$. Hence the non-linear operator	$\widetilde{\mathcal{A}}$ is contractive.
\end{proof}
\begin{thm}\label{thm2}
Let	$\psi_k=\sum_{i=0}^{k}\mu_i$ be $k$ term truncated series solution to the coagulation equation \eqref{agg} for $\mu_1,\mu_2,\ldots,\mu_k$ being the components of the iterative solution. In addition, if the parameter $h$ is chosen in such a way that there exists a constant $\Theta_{h} \in (0,1)$ satisfying
	\begin{align*}
		\Theta_{h} := |1+h|+\gamma |h|,
	\end{align*}
	and $\|\mu_{1}\| < \infty$, then the approximated series solution converges to the exact one with the error bound
	$$\|c(s,\tau)-\psi_k(s,\tau)\|\leq \dfrac{\Theta_{h}^m}{1-\Theta_{h}}\|\mu_1\|,$$
	assuming that all the conditions of Theorem \ref{thm1} hold. 

\end{thm}
\begin{proof}
	Using the coefficients \eqref{iteragg}, the $k+1$-term approximated solution is defined as
	\begin{align*}
		\psi_{k+1}= \mu_{0}+\int_{0}^{\tau}\bigg((1+h)\frac{\pd \psi_k}{\pd \tau} + hA(\psi_{k})\bigg)d\tau.
	\end{align*}
	Now, using the conditions of Theorem \ref{thm1}, we have
	\begin{align*}
		\|\psi_{k+1}-\psi_{m+1}\|\leq& \|(\psi_k-\psi_m)(1+h)\|+ \||h|\int_{0}^{\tau}(c_{0}(s)+A[\psi_k]-A[\psi_m]-c_{0}(s))d\tau\|\\
		\leq& \|(\psi_k-\psi_m)(1+h)\|+ |h|\|\mathcal{A}[\psi_k]-\mathcal{A}[\psi_m]\|\\
		\leq & |1+h|\|\psi_k-\psi_m\|+ \gamma |h| \|\psi_k-\psi_m\|\\
		\leq & \Theta_{h}\|\psi_k-\psi_m\|,
	\end{align*}
	where $\Theta_{h} := |1+h|+\gamma |h|.$ Hence, the following result is observed
	\begin{align*}
		\|\psi_{m+1}-\psi_m\| \leq& \Theta_{h} \|\psi_m-\psi_{m-1}\|\leq \Theta_{h}^2\|\psi_{m-1}-\psi_{m-2}\|\leq \hdots \leq  \Theta_{h}^m\|\psi_1-\psi_0\|.
	\end{align*}
	Using the triangle inequality and the properties of norm for all $k,m\in \mathbb{N}$ with $k>m$, it is certain that
	\begin{align*}
		\|\psi_k-\psi_m\|\leq & \|\psi_{m+1}-\psi_m\|+\|\psi_{m+2}-\psi_{m+1}\|+\cdots+\|\psi_{k}-\psi_{k-1}\|\\
		\leq & (\Theta_{h}^m+ \Theta_{h}^{m+1}+\cdots+ \Theta_{h}^{k-1})\|\psi_1-\psi_0\|\\
		\leq & \Theta_{h}^m\bigg(\dfrac{1-\Theta_{h}^{k-m}}{1-\Theta_{h}}\bigg)\|\mu_1\|,
	\end{align*}
which converges to zero as $m \rightarrow \infty$ as $0< \Theta_{h}< 1$.
 This implies that there exists a $\psi$ such that $\lim\limits_{k\rightarrow \infty }\psi_k=\psi.$ Thus, we have $\psi= \sum_{k=0}^{\infty} \mu_k=c(s,\tau)$, which is the exact solution of the coagulation equation. Further, by fixing $m$ and letting $k \rightarrow \infty$, the desired theoretical error bound is obtained as
	$$\|c(s,\tau)-\psi_k(s,\tau)\|\leq \dfrac{\Theta_{h}^m}{1-\Theta_{h}}\|\mu_1\|.$$ 
\end{proof}
\begin{rem}
	Consider the convergence control parameter $h$ with $\Theta_{h}<1$, so that
	\begin{align*}
		\Theta_{h}=|1+h|+\gamma |h|<1 \implies \gamma < \dfrac{1-|1+h|}{|h|}, \quad h\neq 0.
	\end{align*}
	From the RHS of the above equation, we get
	\begin{align*}
		\dfrac{1-|1+h|}{|h|}=
		\begin{cases}
			-1-\frac{2}{h}, & h<-1,\\
			1 ,& -1\leq h<0,\\
			-1,& h>0.
		\end{cases}
	\end{align*}
	Therefore, one can choose the parameter $h \in [-1,0).$  
\end{rem}
\subsection{AHAM implementation for coupled aggregation-breakage equation (CABE) } 
This section delves the mathematical formulation of AHAM for solving CABE \eqref{agg-brk} and its convergence analysis.
On comparing \eqref{agg-brk} and \eqref{7}, we have $g(x,\tau)=0$, and
\begin{align*}
L[c(s,\tau)]=-\int_{s}^{\infty}\beta(s,\xi)\mathcal{S}(\xi)c(\xi,\tau)d\xi+\mathcal{S}(s)c(s,\tau),
\end{align*} 
and
\begin{align*}
	 M[c(s,\tau)]=-\frac{1}{2}\int_{0}^{s}w(s-\xi,\xi)c(s-\xi,\tau)c(\xi,\tau)d\xi+\int_{0}^{\infty}w(s,\xi)c(s,\tau)c(\xi,\tau)d\xi.
\end{align*}
  
Following the procedure defined in Section 3.2, iterations \eqref{iter} are as follows
\begin{align}\label{iter3}
	 \mu_{0}=&c_{0}(s), \nonumber \\
	\mu_{1}=&\int_{0}^{\tau}h\bigg[\frac{\partial \mu_{0}}{\partial\rho}+L\mu_{0}(s,\rho) +H_{0}\bigg]{d\rho}, \nonumber \\
	\mu_{2}=&\int_{0}^{\tau}\bigg(\frac{\partial \mu_{1}}{\partial\rho}+h\bigg[\frac{\partial \mu_{1}}{\partial\rho}+L\mu_{1}(s,\rho) +H_{1}\bigg]\bigg) {d\rho},\nonumber \\ 
	&\vdots \nonumber \\
	\mu_{k}=&\int_{0}^{\tau}\bigg(\frac{\partial \mu_{k-1}}{\partial\rho}+h\bigg[\frac{\partial \mu_{k-1}}{\partial\rho}+L\mu_{k-1}(s,\rho)+ H_{k-1}\bigg]\bigg)  {d\rho}.
\end{align}
These series solution components include the convergence control parameter, which is determined using the discrete square residual error approach mentioned in Remark (2.1). Then the solution to the problem \eqref{agg-brk} is presented as
\begin{align*}
	c(s,\tau)=\lim_{r \to 1} \Psi(s,\tau,r) = \sum_{k=0}^{\infty}\mu_{k}(s,\tau),
\end{align*}
and for approximation, the $k$-th order truncated series solution is provided  as $\psi_{k}=\sum_{i=0}^{k}\mu_{i}(s,\tau).$
\subsection{Convergence analysis for CABE } 
To study the convergence analysis of series solution obtained via AHAM,
 let us consider the space $\mathbb{X}_{2}= (\mathcal{C}(\mathcal{L}^{1}[0 , \infty):[0 ,T]), \|.\|)$ to be a Banach space with norm 
\begin{align*}
	\|c\|	= \sup_{\tau\in [0,\tau_{0}]} \int_{0}^{\infty} \exp(\sigma s)\left|c(s,\tau)  \right|ds < \infty, 
\end{align*}
where $\sigma >0. $
Now, take the equation \eqref{agg-brk} in the operator form as
\begin{align}\label{20}
	c(s,\tau) = \mathcal{V}[c],
\end{align}
where $\mathcal{V}$ is a non-linear operator defined over $\mathbb{X}_{2}$ and is given by
\begin{align}\label{21}
\mathcal{V}[c]=& c_{0}(s)+ \int_{0}^{\tau}V(c)d\rho.
\end{align}
where $$ V(c)= \frac{1}{2}\int_{0}^{s}w(s-\xi,\xi)c(s-\xi,\rho)c(\xi,\rho)d\xi-\int_{0}^{\infty}w(s,\xi)c(s,\rho)c(\xi,\rho)d\xi+\int_{s}^{\infty}\beta(s,\xi)\mathcal{S}(\xi)c(\xi,\rho)d\xi -\mathcal{S}(s)c(s,\rho) $$
To show the contraction behavior of the operator $\mathcal{V}$, an equivalent form of the equation \eqref{21} can be written as
\begin{align*}
	\frac{\partial }{\partial\tau}\bigg(c(s,\tau)\exp[R(s,\tau,c)]\bigg)=\exp[R(s,\tau,c)]\bigg(\frac{1}{2}\int_{0}^{s}w(s-\xi,\xi)c(s-\xi,\tau)c(\xi,\tau)d\xi+\int_{s}^{\infty}\beta(s,\xi)\mathcal{S}(\xi)c(\xi,\tau)d\xi\bigg),
\end{align*}
where 
\begin{align}\label{D}
	R(s,\tau,c)=S(s)\tau + \int_{0}^{\tau}\int_{0}^{\infty}w(s,\xi)c(\xi,\rho)d\xi d\rho.
\end{align}
Thus, we have $$ c=\widetilde{\mathcal{V}} c,$$ 
where
\begin{align}\label{22}
	\widetilde{\mathcal{V}}c=&c_{0}\exp[-R(s,\tau,c)]+\int_{0}^{\tau}\exp[R(s,\rho,c)-R(s,\tau,c)]\bigg(\int_{0}^{s}\frac{1}{2}w(s-\xi,\xi)c(s-\xi,\rho)c(\xi,\rho)d\xi \nonumber \\& +\int_{s}^{\infty}\beta(s,\xi)\mathcal{S}(\xi)c(\xi,\rho)d\xi\bigg) d\rho.
\end{align}
	\begin{thm}\label{thm3}
	The operator $\widetilde{\mathcal{V}}$  is contractive on $\mathbb{X}_{2}$, i.e, $\|\widetilde{\mathcal{V}}c-\widetilde{\mathcal{V}}c^{\ast}\| \leq \nabla \|c-c^{\ast}\|$, $\forall$ $c,c^{\ast}\in\mathbb{X}_{2}$ with 
	\begin{itemize}
		\item[1] $\beta(s,\xi)=\frac{\eta s^{i-1}}{\xi^i}$, $i=1,2,\ldots$ and $\eta>0$ is a constant satisfying $\int_{0}^{\xi}s \beta(s,\xi)ds = \xi$,
			\item[2] $S(s)\leq s^{j}, j= 1,2,\ldots,$
				\item[3] $ \sigma $ is taken such that $[\exp(\sigma \xi)-1]<1,$
					\item[4] $w(s,\xi)=1$ for all $s,\xi \in(0,\infty)$ and, 
		\item[5] $ \gamma_{2}= \widetilde{T} e^{2\widetilde{T} D}( \|c_{0}\|+ 2D(\tau_{0}D+1)+(2\tau_{0}+1)\frac{\eta(j-1)!}{\sigma^j})< 1$ where $D=\|c_{0}\|(\tau+1)  $ for suitable $\tau_{0}$
	\end{itemize}
\end{thm}
\begin{proof}
	Take $c,c^{\ast}$ from the space $\mathbb{X}_{2}$ and we get
\begin{align}\label{2co}
	\widetilde{\mathcal{V}}[c]-	\widetilde{\mathcal{V}}[c]^{\ast}= & c_{0}(s) \mathcal{U}_{1}(s,0,\tau)+\int_{0}^{\tau}\mathcal{U}_{1}(s,\rho,\tau)\bigg[\frac{1}{2}\int_{0}^{\xi}c(s-\xi,\xi)c(\xi,\rho)d\xi+\int_{s}^{\infty}\beta(s,\xi)S(\xi)c(\xi,\rho)d\xi\bigg] d\rho \nonumber \\&-\int_{0}^{\tau}\exp[R(s,\rho,c^{\ast})-R(s,\tau,c^{\ast})] \bigg[\frac{1}{2}\int_{0}^{s} c^{\ast}(s-\xi,\xi)[c^{\ast}(\xi,s)-c(\xi,s)] d\xi  \nonumber \\&+\int_{0}^{s} \frac{1}{2}c(\xi,\rho)[c^{\ast}(s-\xi,\rho)-c(s-\xi,\rho)] d\xi +\int_{s}^{\infty}\beta(s,\xi)S(\xi)(c^{\ast}(\xi,\rho)-c(\xi,\rho))d\xi\bigg] d\rho,
\end{align}

here $$\mathcal{U}_{1}(s,\rho,\tau) =\exp[R(s,\rho,c)-R(s,\tau,c)]-\exp[R(s,\rho,c^{\ast})-R(s,\tau,c^{\ast})]. $$ It can be easily shown that
\begin{align*}
\left|	\mathcal{U}_{1}(s,\rho,\tau)\right|\leq \delta_{1} \|c-c^{\ast}\|,
\end{align*}
where, $\delta_{1} = \tau\exp[\tau J_{1}] $ and $J_{1} = max{\{\|c\|,\|c^{\ast}\|\}}.$
To show the contraction of the operator $\widetilde{V}$, consider a set $Q_{1}=\{c\in \mathbb{X}_{2}: \|c\|\leq 2D\}.$
 It can be seen that $\widetilde{V}$  maps $Q_{1}$ into itself. For $c,c^{\ast}\in Q_{1} $ we have $J_{1} \leq 2M$. By taking norm on both side of equation \eqref{2co}
Therefore, we have 
\begin{align*}
\|	\widetilde{\mathcal{V}}[c]-\widetilde{\mathcal{V}}[c]^{\ast}\|\leq& \delta_{1}\|c_{0}\|\|c-c^{\ast}\|+\delta_{1}\|c-c^{\ast}\|\int_{0}^{\tau}(\frac{1}{2}\|c\|^2 +\frac{\eta(j-1)!}{\sigma^j}\|c\|)d\rho\\&+ \int_{0}^{\tau}\exp\{(\tau)J_{1}\}[\frac{1}{2}(\|c\|+\|c^{\ast}\|)\|c-c^{\ast}\|+ \frac{\eta(j-1)!}{\sigma^j}\|c-c^{\ast}\|]d\rho \\ 
	\leq& \delta_{1}\bigg[ \|c_{0}\|+\tau(\|\frac{1}{2}c\|^2+\frac{d(j-1)!}{\sigma^j}\|c\|)+\frac{1}{2}(\|c\|+\|c^{\ast}\|)+\frac{\eta(j-1)!}{\sigma^j}\bigg]\|c-c^{\ast}\|
\end{align*}
The non-linear operator $\widetilde{\mathcal{V}}$ is contractive if $ \gamma_{2}= \widetilde{T} e^{2\widetilde{T} D}( \|c_{0}\|+ 2D(\tau_{0}D+1)+(2\tau_{0}+1)\frac{\eta(j-1)!}{\sigma^j})< 1$.
\end{proof}
\begin{thm}\label{thm4}
	Let $\widetilde{\mathcal{V}}$ be the operator given by \eqref{22}, satisfying all the conditions outlined in Theorem \ref{thm3}.
	Let $\mu_1,\mu_2,\ldots,\mu_k$ are the components of the series solution and $\psi_k=\sum_{i=0}^{k}\mu_i$ be $k$ term truncated series solution. In addition, if $h$ is chosen in such a way that there exist a constant $\Theta_{2} \in (0,1)$ satisfying
	\begin{align*}
		\Theta_{2} := |1+h|+\lambda |h|,
	\end{align*}
	and $\|\mu_{1}\| < \infty$.  Then the approximated series solution converges to the exact one with the error bound
	$$\|c(s,\tau)-\psi_k(s,\tau)\|\leq \dfrac{\Theta_{2}^m}{1-\Theta_{2}}\|\mu_1\|,$$
under the assumption that all the conditions of Theorem \ref{thm3} are satisfied.	
\end{thm}
\begin{proof}
	Using \eqref{iter3}, $k+1$-term approximated series solution is defined as
	\begin{align*}
		\psi_{k+1}= \mu_{0}+\int_{0}^{\tau}\bigg((1+h)\frac{\pd \psi_k}{\pd \tau} +hV(\psi_{k})\bigg)d\tau.
	\end{align*}
	
	Now, using the conditions of Theorem \ref{thm3}, we have
	\begin{align*}
		\|\psi_{k+1}-\psi_{m+1}\|\leq& \|(\psi_k-\psi_m)(1+h)\|+ |h|\|\mathcal{V}[\psi_k]-\mathcal{V}[\psi_m]\|\\
		\leq & |1+h|\|\psi_k-\psi_m\|+ \delta |h| \|\psi_k-\psi_m\|\\
		\leq & \lambda\|\psi_k-\psi_m\|,
	\end{align*}
	where $\Theta_{2} := |1+h|+\lambda|h|.$ Hence, the following results is observed
	\begin{align*}
		\|\psi_{m+1}-\psi_m\| \leq& \Theta_{2} \|\psi_m-\psi_{m-1}\|\leq \Theta_{2}^2\|\psi_{m-1}-\psi_{m-2}\| \leq \ldots\leq  \Theta_{2}^m\|\psi_1-\psi_0\|.
	\end{align*}
	Using the properties of norm for all $k,m\in \mathbb{N}$ with $k>m$, it is certain that
	\begin{align*}
		\|\psi_k-\psi_m\|\leq & \|\psi_{m+1}-\psi_m\|+\|\psi_{m+2}-\psi_{m+1}\|+\cdots+\|\psi_{k}-\psi_{k-1}\|\\
		\leq & (\Theta_{2}^m+ \Theta_{2}^{m+1}+\cdots+ \Theta_{2}^{k-1})\|\psi_1-\psi_0\|\\
		\leq &  \Theta_{2}^m\bigg(\dfrac{1-\Theta_{2}^{k-m}}{1-\Theta_{2}}\bigg)\|\mu_1\| 
		\leq \dfrac{\Theta_{2}^m}{1-\Theta_{2}}\|\mu_1\|.
	\end{align*}
One can see that $ 	\|\psi_k-\psi_m\|\to 0 $ as $n\to\infty$ and $\Theta_{2}<1.$
This implies that there exists a $\psi$ such that $\lim\limits_{k\rightarrow \infty }\psi_k=\psi.$ Thus, we have $\psi= \sum_{k=0}^{\infty} \mu_k=c(\tau,s)$, which is the exact solution of the coupled aggregation-breakage equation. Further, by fixing $m$ and letting $k \rightarrow \infty$, error bound is obtained as
	$$\|c(s,\tau)-\psi_k(s,\tau)\|\leq \dfrac{\Theta_{2}^m}{1-\Theta_{2}}\|\mu_1\|.$$ This concludes the proof of the theorem.
\end{proof}
	\section{Numerical Testing} 
In this section, we have examined the method for five physical test cases of non-linear aggregation equation \eqref{agg} and findings of number density and moments are compared to the analytical solutions and other existing methods. Thanks to the noticeable and better outcomes in pure aggregation equation \eqref{agg}, the proposed algorithm is extended to solve the CABE \eqref{agg-brk}. Two test scenarios for the model are being considered to demonstrate the effectiveness of our proposed scheme.
				\begin{example}\label{q1} 
Consider the aggregation equation \eqref{agg} with the aggregation kernel $w(s,\xi)=1$, and an initial condition in the form of inverse exponential function, i.e;  $c_{0}(s)=e^{-s}$. The constant kernel helps analyze early phases of coagulation when particles are small, as in the case of protein aggregation. The precise solution of this problem is presented in \cite{ranjbar2010numerical} as
			\begin{align*}
				c (s, \tau)=4 \frac{e^{\frac{-2s}{\tau+2}}}{(\tau+2)^2}.
			\end{align*}
			Using the equation (\ref{iteragg}), first few iterative solutions are obtained as follows,
			\begin{align*}
				\mu_{0} =& e^{-s},\quad
				\mu_{1} = h \tau e^{-s} \left(1-\frac{ s}{2}\right),\\
				\mu_{2} =&h \tau e^{-s}\bigg[1-\frac{s}{2}+h-\frac{hs}{2}+h\tau(-\frac{3s}{4}+\frac{3}{4}+\frac{s^2}{8})+h^2\tau^2(\frac{1}{6}-\frac{s}{4}+\frac{s^2}{12}-\frac{s^3}{144})\bigg].
\end{align*}
\end{example}
Due to the complexity of the iterations, only a few of them are presented here. For greater precision, more iterations can be performed with the help of MATHEMATICA. For the numerical study, third term iterative solution is considered and the outcomes are compared with the exact solution and the iterative solutions obtained using existing HAM \cite{kaur2022approximate}, HPM \cite{kaur2019analytical}, ADM \cite{singh2015adomian}, ODM \cite{kaushik2023novel} and LODM \cite{kaushik2023laplace}, see Figure \ref{fig2}. Additionally the absolute errors between exact and approximated solutions are also reported. It is essential to emphasize that the HAM, ADM, and HPM all yield same iterations and consequently, the identical finite term series solutions.

		\begin{figure}[htb!]
	\centering
	\subfigure[Number density at $\tau$ = 2 ]{\includegraphics[width=0.40\textwidth]{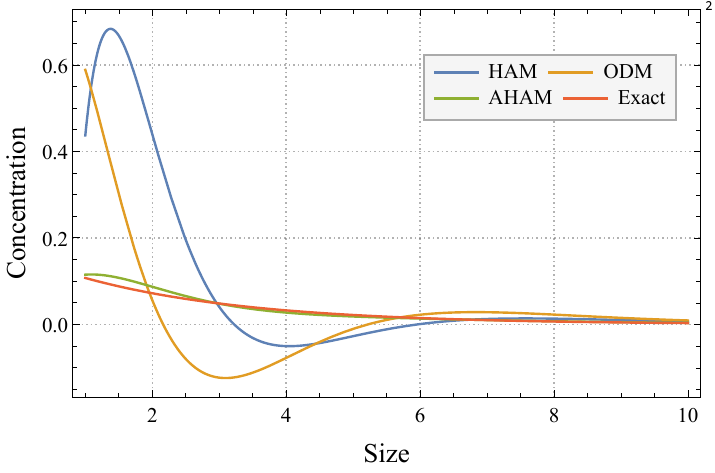}}
	\subfigure[Error at $s$ = 5 ]{\includegraphics[width=0.40\textwidth]{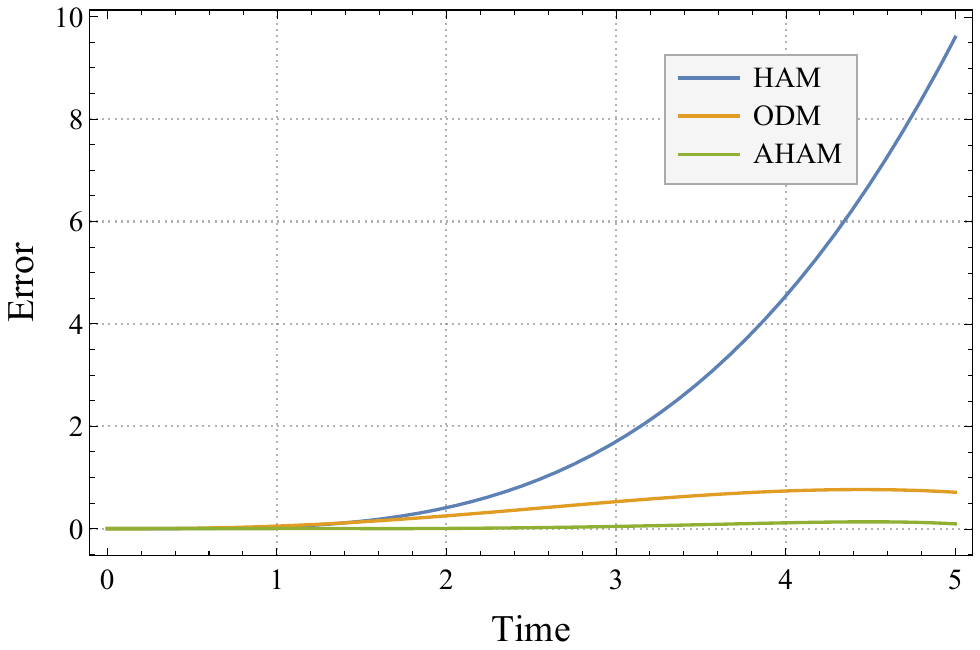}}
\caption{Number density and error plots for Example \ref{q1}}
	\label{fig2}
\end{figure}
 It is observed that HAM and ODM solutions show disparities compared to the exact density function, whereas AHAM solutions maintain accuracy. The findings substantiate the noteworthy concordance between the AHAM and the exact solution. The error results presented in Figure \ref{fig2}(b) demonstrate that while all the schemes exhibit comparatively low error for a short time, over extended periods the errors associated with HAM and ODM are considerably higher than AHAM. Additionally, Table \ref{table1} compares the absolute errors at $s=5$ for the HAM, ODM, and AHAM methods. 
 \begin{table}
	\begin{center}
		\begin{tabular}{ p{0.5cm}| p{2.2cm} p{2.2cm} p{2.2cm} p{2.2cm} p{2.2cm} p{2.2cm} p{2.5cm}  }
			\multirow{2}{*} {$\tau$}  
	
			& Exact & HAM & ODM & AHAM & HAM Error & ODM Error & AHAM Error \\ \hline
			& & & & & & &\\ 
			0.5 & 1.1722E-2  & 1.16686E-2 & 1.17467E-2 & 1.17527E-2 &5.34288 E-5 & 2.47063 E-5 &  3.06568 E-5\\ 
			& & & & & & & \\ 
			1 &1.58551E-2 & 1.502E-2& 1.56451E-2 &1.59766E-2 & 8.35101E-4 & 2.1002E-4 &  1.21521 E-4 \\
			& & & & & & & \\ 
			1.5 &1.87535E-2 & 1.47919E-2  &1.69015E-2  &1.8747 E-2  & 3.96161E-3  & 1.85196 E-3 &  6.55172E-6 \\
			& & & & & & & \\ 
			2.0 & 2.05212 E-2 & 8.98393E-3  & 1.39846E-2  & 1.9909E-2  & 1.15373E-2 &  6.53667E-3 & 6.12241 E-4 \\
			& & & & & & & \\ 
			2.5 & 2.1406E-2 & -4.4042E-3  & 5.36266E-3  & 1.96291E-2  &2.58103E-2 &  1.60434E-2 & 1.77691 E-3 \\
			& & & & & & & \\ 
			3.0 & 2.16536 E-2 & -2.73729E-2  & -1.04957E-2  & 1.81253E-2  & 4.90266E-2 &  3.21494E-2 & 3.52839 E-3 \\
			& & & & & & & \\ 
			\hline
		\end{tabular}
	\end{center}
	\caption{Absolute errors for $s = 5$ at different time levels Example 4.1.}
	\label{table1}
\end{table}
Superiority of AHAM can also be seen from Table \ref{table1} that compares the approximated and exact solutions for all the methods at various time. The results indicate that for small time upto $\tau=1$, all methods exhibit good predication with the exact number density but as time progresses, errors due to HAM and ODM start blowing up. Further, it is noticed that ODM and HAM start deviating from the basic facts that solutions are non-negatives see the results at $\tau=2.5$ and $\tau=3$.  
\begin{figure}[htb!]
	\centering
	\subfigure[AHAM error $(n=3)$]{\includegraphics[width=0.32\textwidth]{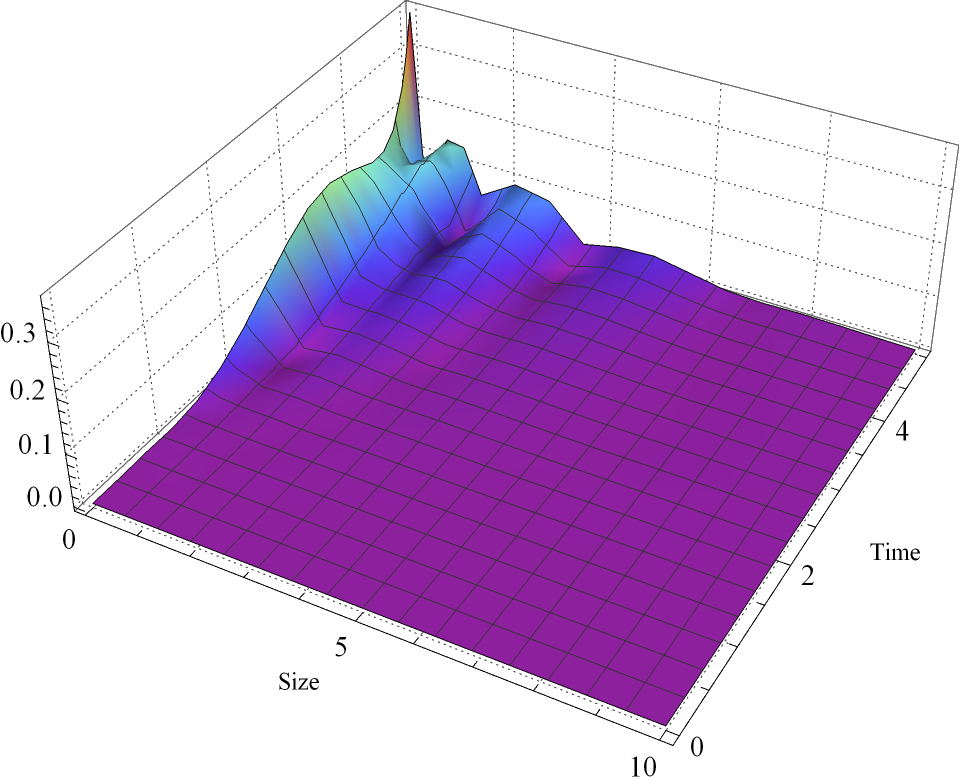}}
	\subfigure[HAM error $(n=3)$ ]{\includegraphics[width=0.32\textwidth]{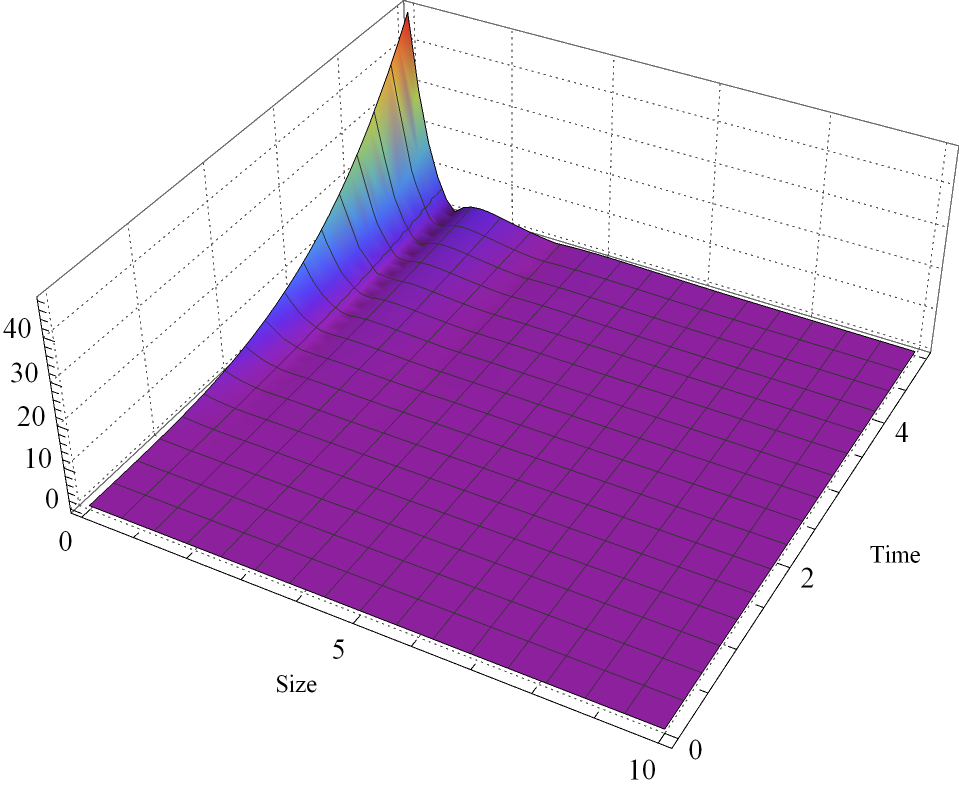}}
	\subfigure[ODM error $(n=3)$]{\includegraphics[width=0.32\textwidth]{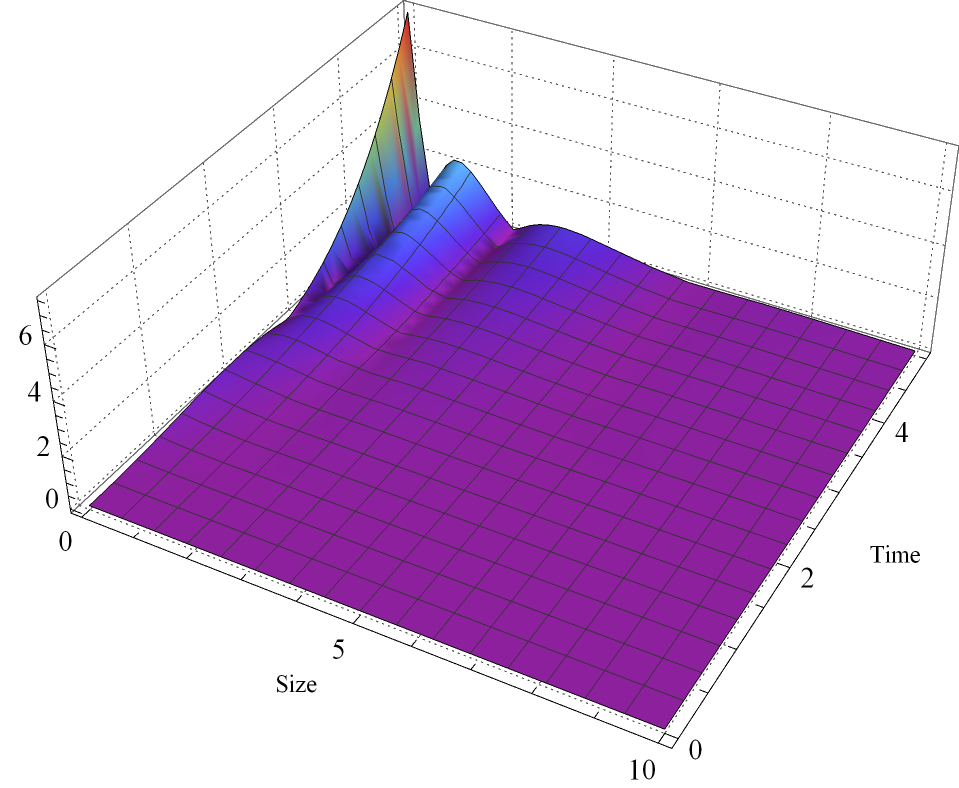}}	
	\caption{Absolute errors for Example \ref{q1}}
	\label{fig3}
\end{figure}
To justify more, Figure \ref{fig3} illustrates 3-dimensional error plots for the iterative methods under consideration. The results depicted in the figure reinforce the observations made in Figure \ref{fig2} (b) and Table \ref{table1}.
	\begin{figure}[htb!]
	\centering
	\subfigure[Zeroth moment]{\includegraphics[width=0.40\textwidth]{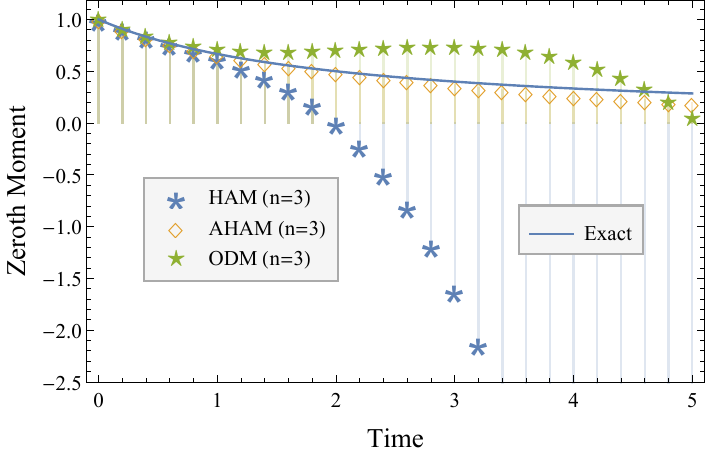}}
	\subfigure[Second moment ]{\includegraphics[width=0.40\textwidth]{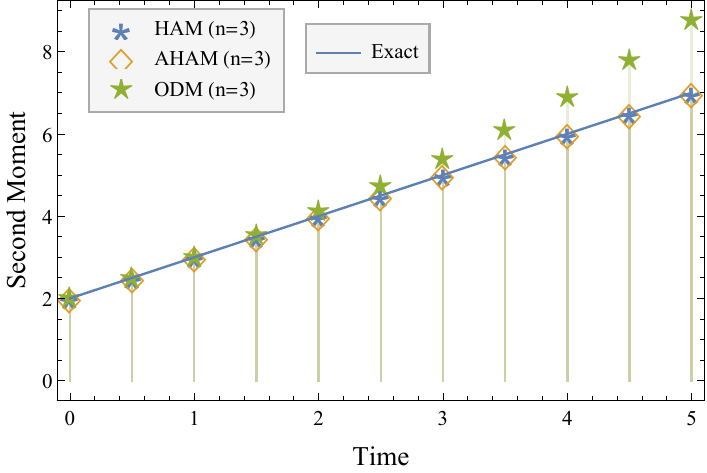}}
	\caption{Zeroth and second moments for Example \ref{q1}}
	\label{fig4}
\end{figure}
As mentioned earlier, zeroth, first and second moments play a vital role in comprehending the physical characteristics of the model. These moments are displayed and comparison are made with the precise ones. The zeroth moments shown in Figure \ref{fig4}(a) demonstrate that the number of particles in a coagulation system decreases with time.
Also, it can be observed that AHAM provides more accurate approximations as compared to HAM which tends to underpredict the result and deviates exponentially from the exact one. Although ODM offers a much better approximation than HAM, it still suffers from the deviations. As depicted in Figure \ref{fig4}(b), AHAM stands out as the superior option, as both AHAM and HAM align closely with the exact second moment, while ODM fails to provide a satisfactory estimate.  Here, it is pertinent to mention that the system preserves the mass, so the first moment is constant throughout the time, hence the plot is omitted.


	\begin{example}\label{q2}
Let us assume aggregation equation \eqref{agg} with sum kernel $w(s,\xi)=s+\xi$ and exponential initial condition $c_{0}(s)=e^{-s}$. The sum kernel has extensive applications in the pharmaceutical industry, particularly for modeling granulation and dairy processes. The exact number density is delineated in \cite{scott1968analytic} as
	\begin{align*}
		c (s, \tau)=\frac{e^{(e^{-\tau}-2)s-\tau} I_{1}(2\sqrt{1-e^{-\tau}}s)}{\sqrt{1-e^{-\tau}}s},
	\end{align*}
 where $I_{1}$ is the modified Bessel's function of the first kind.
	Thanks to equation (\ref{iteragg}), first few iterations are obtained as follows  
	\begin{align*}
		\mu_{0} =& e^{-s},\quad
		\mu_{1} =h \tau e^{-s} \left(s+1-\frac{s^2}{2} \right) ,\\
		\mu_{2} =&\frac{1}{720} h \tau e^{-s} \bigg(h^2 \tau^2 s (s (120-s (s ((s-10) s-20)+240))+240)+60 h \bigg(\tau \left((s-6) s \left(s^2-3\right)+6\right)-6 (s-2) s\\&+12\bigg)-360 ((s-2) s-2)\bigg).
	\end{align*}
\end{example}
The coefficients become complicated as we go along, but one can use MATHEMATICA to calculate the higher terms using Eq. \ref{iteragg}.
\begin{figure}[htb!]
	\centering
	\subfigure[Number density at $\tau$ = 2 ]{\includegraphics[width=0.40\textwidth]{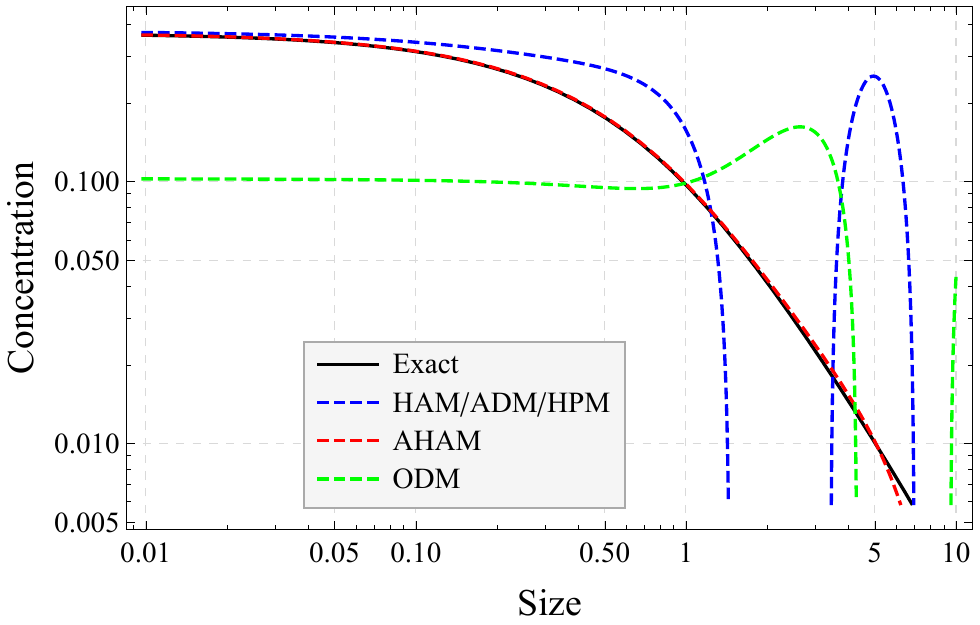}}
	\subfigure[Error at $s$ = 5 ]{\includegraphics[width=0.40\textwidth]{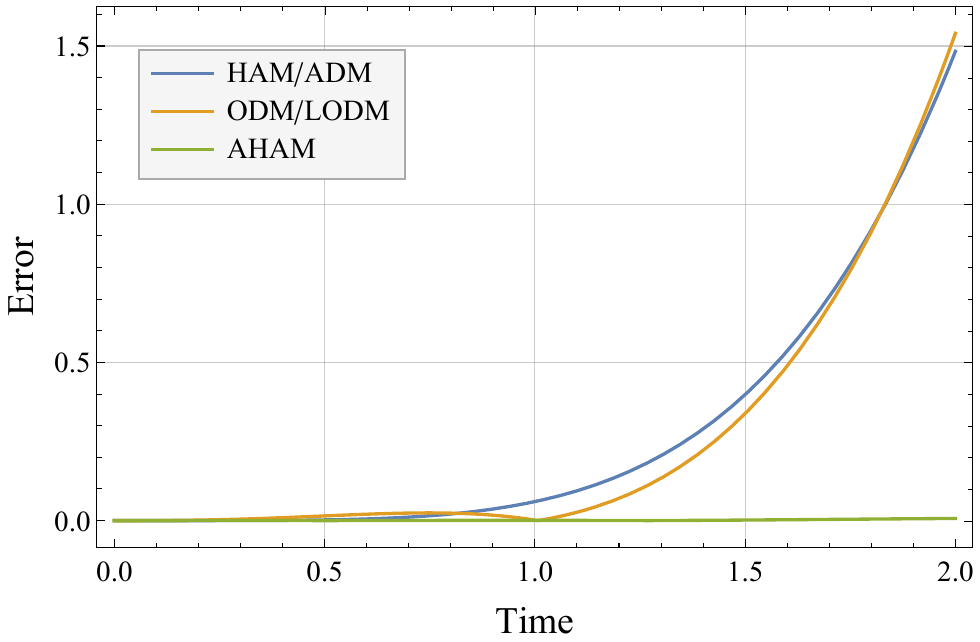}}
	\caption{Number density and error plots for Example \ref{q2}}
	\label{fig5}
\end{figure}
Figure \ref{fig5}(a) illustrates the number density using four-term approximate solutions across various schemes with the exact solution. The figure highlights that the HAM and ODM solutions display deviation compared to the precise solution. In contrast, the AHAM solution precisely aligns with the exact one. Additionally, as depicted in Figure \ref{fig5}(b), the errors derived from HAM and ODM are nearly identical and negligible for a constant value of $s$ and up to $\tau = 1$. However, with the progression of time, the errors generated by these two schemes start to escalate almost exponentially while AHAM consistently performs well.
	\begin{figure}[htb!]
	\centering
	\subfigure[AHAM error $(n=4)$]{\includegraphics[width=0.32\textwidth]{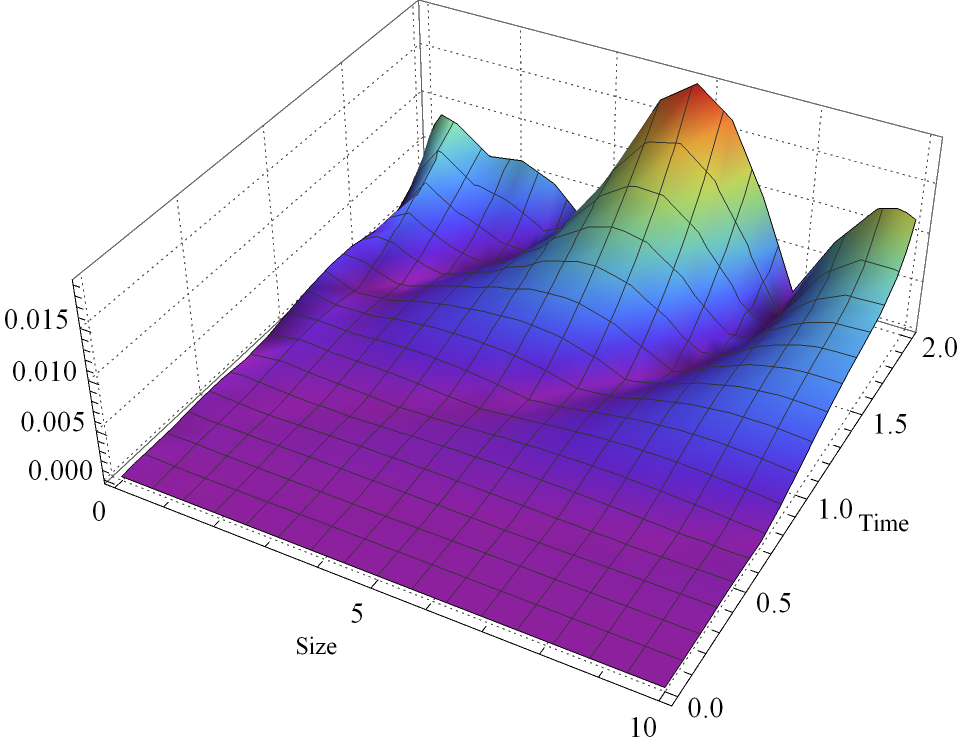}}
	\subfigure[HAM error $(n=4)$]{\includegraphics[width=0.32\textwidth]{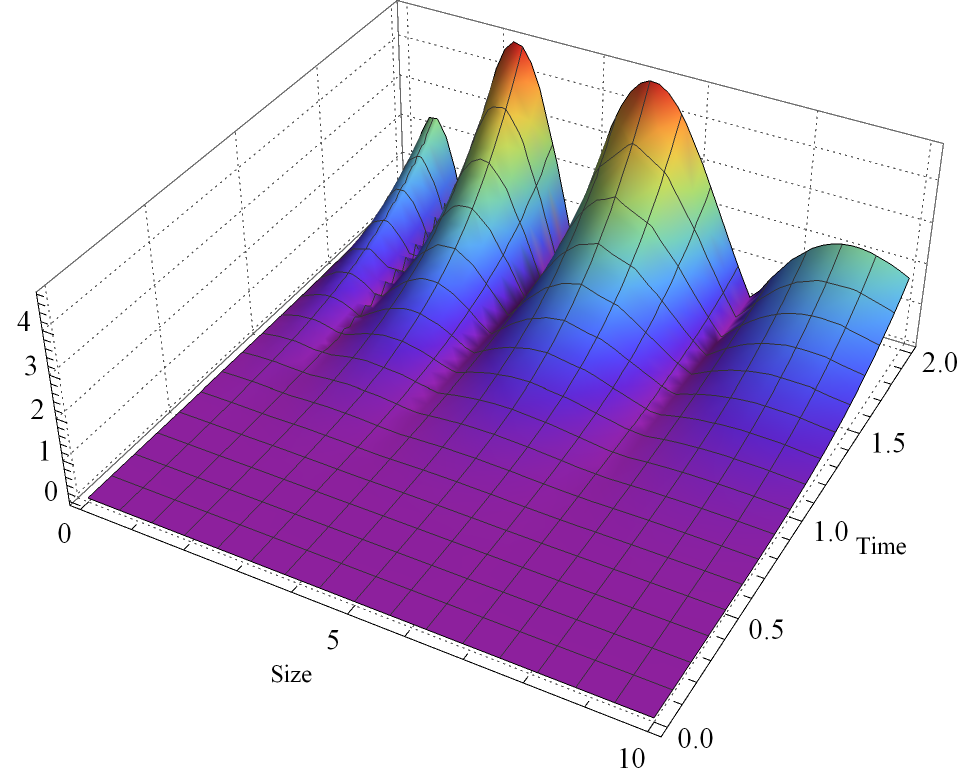}}
	\subfigure[ODM error $(n=4)$]{\includegraphics[width=0.32\textwidth]{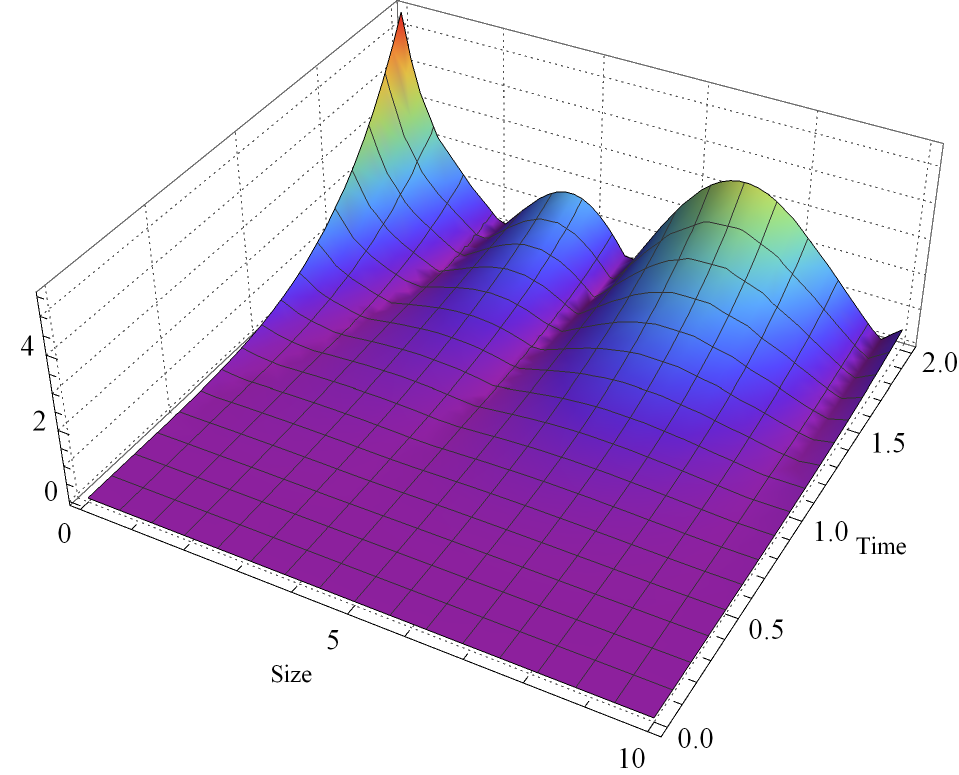}}	
	\caption{AHAM, HAM and ODM errors for Example \ref{q2}}
	\label{fig6}
\end{figure}
Further in Figure \ref{fig6}, 3-dimensional representations of the errors are presented, displaying that the proposed AHAM delivers exceptional accuracy over larger time intervals as compared to the HAM as well as ODM.
	\begin{figure}[htb!]
	\centering
	\subfigure[Zeroth moments]{\includegraphics[width=0.40\textwidth]{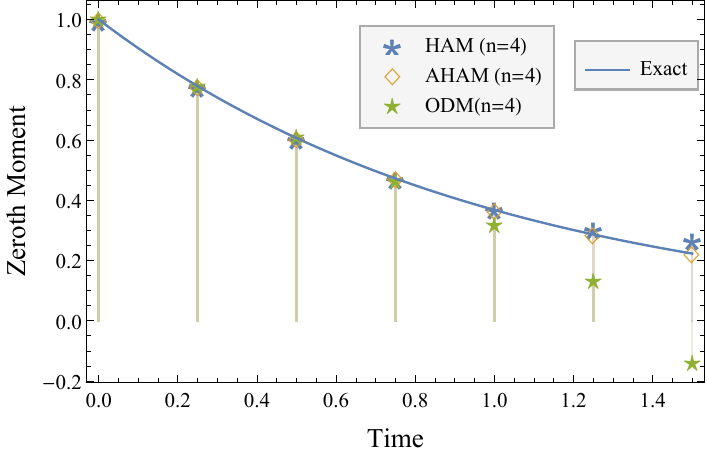}}
	\subfigure[Second moments ]{\includegraphics[width=0.40\textwidth]{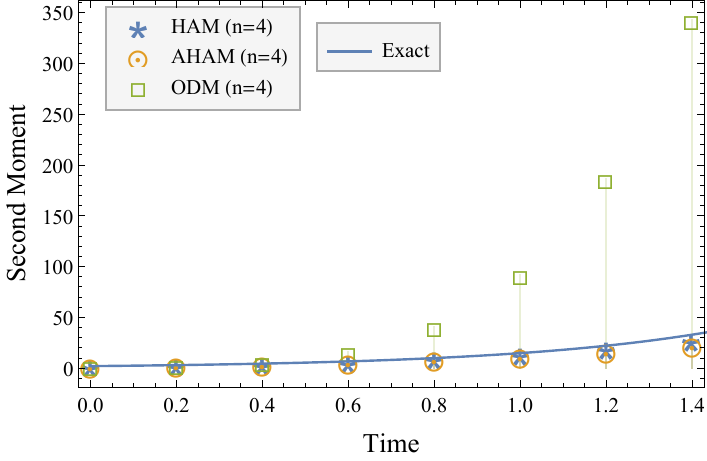}}
	\caption{Zeroth and second moments plots for Example \ref{q2}}
	\label{fig7}
\end{figure}
In Figure \ref{fig7}, the exact and approximated moments calculated via all three methods are compared. It is observed that ODM under-predicts the zeroth moment and over-predicts the second moment. In contrast, both HAM and AHAM provide almost identical results, delivering an excellent approximation to both the precise moments.

	\begin{example}\label{q3}
Taking the aggregation equation \eqref{agg} with $c_{0}=e^{-s}$ and product kernel $w(s, \xi) = s\xi$, the exact solution is given as  
	\begin{align*}
			c (s, \tau)=e^{-(1 + \tau)s} \sum_{k=0}^{\infty}\frac{\tau^k s^{3k}}{(k+1)!\Gamma(2k+2)}.
	\end{align*}
The product kernel facilitates research on coagulation behavior at various particle sizes, including colloidal suspensions. It is worth noting that the product kernel is important as a physically relevant kernel, especially in dairy sciences, where it is used to simulate the rennet-induced coagulation process. \\
	Following the recursive scheme defined in equation \eqref{iteragg}, the first few components of the AHAM solution are calculated as
	\begin{align*}
		\mu_{0} =& e^{-s} ,\quad
		\mu_{1} =h \tau e^{-s}s \left( 1-\frac{s^3}{12} \right) ,\\
		\mu_{2} =&-\frac{1}{544320}h \tau e^{-s} s \bigg(h^2 \tau^2 \left(s^4-144 s^2+3024\right) s^4-756 h \left(s \left(\tau \left(s^4-60 s^2+360\right)-60 s\right)+720\right)\\&+45360 \left(s^2-12\right)\bigg).
	\end{align*}
\end{example}
In this instance, numerical experiments demonstrate high accuracy over extended time when a three-term iterative solution is used.
\begin{figure}[htb!]
	\centering
	\subfigure[Number density at $\tau$ = 2]{\includegraphics[width=0.40\textwidth]{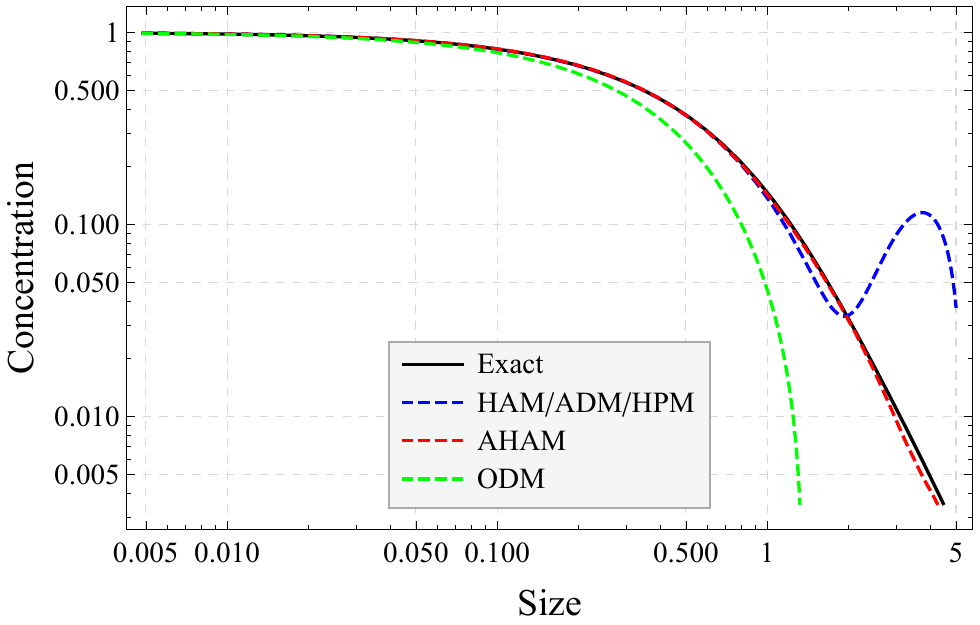}}
	\subfigure[Error at $s$ = 5 ]{\includegraphics[width=0.40\textwidth]{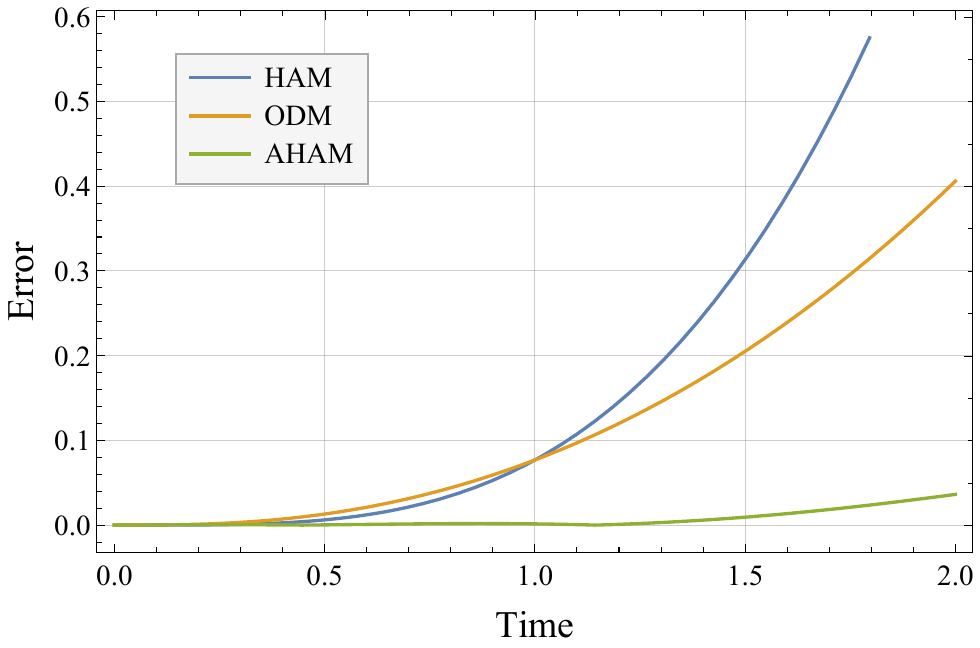}}
	\caption{ Number density and error plots  for Example \ref{q3}}
	\label{fig8}
\end{figure} 
  In Figure \ref{fig8}(a), the concentration of particles at time $\tau = 2$ is represented, indicating that solutions obtained through HAM and ODM experience blow-up, whereas the AHAM solution matches closely with the exact number density. Figure \ref{fig8}(b) shows analogous observations for the absolute error.
\begin{figure}[htb!]
	\centering
	\subfigure[Zeroth moments]{\includegraphics[width=0.40\textwidth]{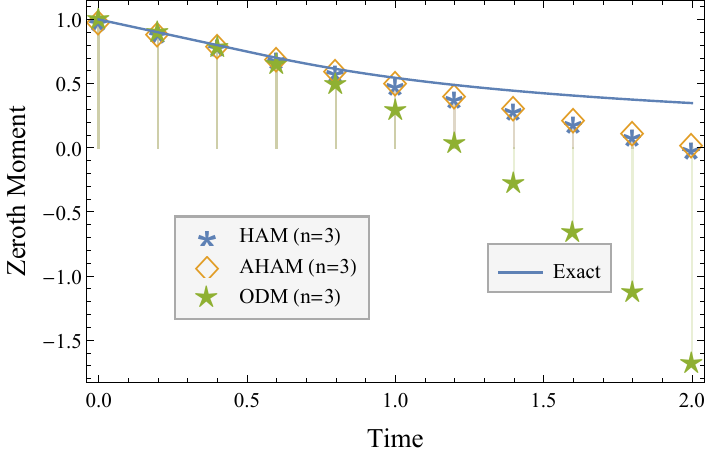}}
	\subfigure[Second moments ]{\includegraphics[width=0.40\textwidth]{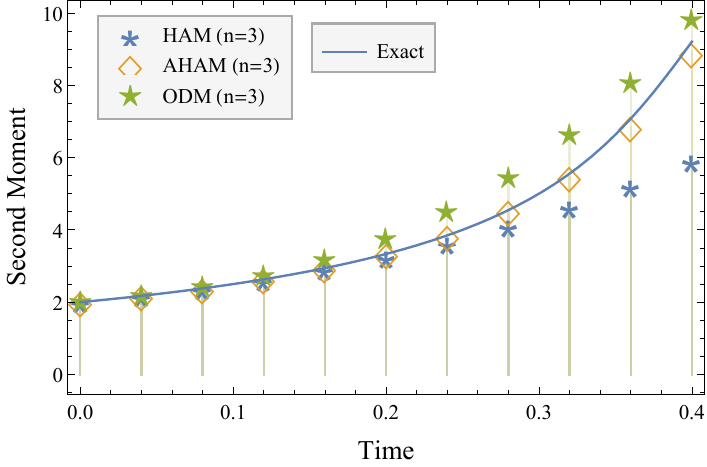}}
	\caption{Zeroth and second moments for Example \ref{q3}}
	\label{fig9}
\end{figure}
Moving further, Figure \ref{fig9} compares the analytical and estimated moments. As observed in the previous cases, ODM moments explode, whereas the moments derived from HAM and AHAM are closely aligned and converging toward the precise moments.
\begin{example}\label{q4}
		Here we are considering a Gaussian-like distribution, i.e, gamma initial distribution defined by 
		\begin{align}\label{key}
		c_{0}(s)= (1+\nu)^{1+\nu} s^{\nu}\frac{e^{-(1+\nu)s}}{\sqrt{1+\nu}}
	\end{align} with sum kernel.
For simplicity, we are assuming the value of $\nu$ is 1, but any other value of $\nu$ can also be considered. The exact solution of this problem is available in \cite{scott1968analytic}.
 Using the iterative scheme defined in equation \ref{iteragg}, iterative solutions are obtained as follows
	\begin{align*}
		\mu_{0} =& 4se^{-2s} ,\quad
		\mu_{1} =4h \tau e^{-2 s}s \left(s+1-\frac{s^3}{3}\right) ,\\
		\mu_{2} =&\frac{2 h \tau e^{-2 s} s}{8505} \bigg(-2 h^2 \tau^2 s \left(s \left(s \left(s \left(s \left(s \left(s^3-36 s-126\right)+189\right)+1890\right)+945\right)-2835\right)-2835\right)\\&+189 h \tau \left(s \left(s \left(s \left(2 s^3-30 s-45\right)+45\right)+135\right)+45\right)-5670 (h+1) \left(s^3-3 s-3\right)\bigg).
	\end{align*}
\end{example}
For the numerical experiments, third order iterative solution is taken into consideration. 
\begin{figure}[htb!]
	\centering
	\subfigure[Number denisty ]{\includegraphics[width=0.40\textwidth]{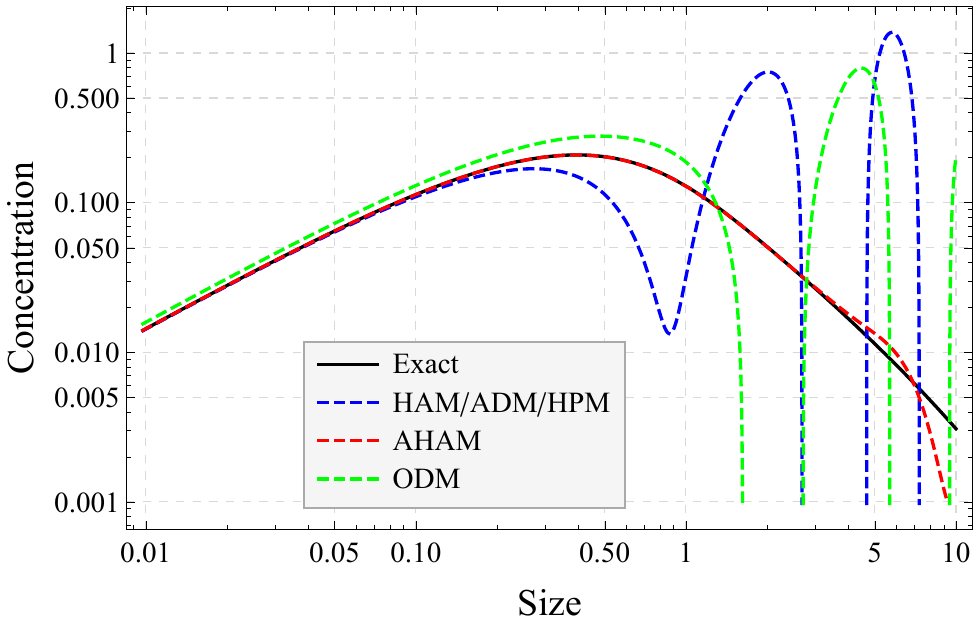}}
	\subfigure[Error ]{\includegraphics[width=0.40\textwidth]{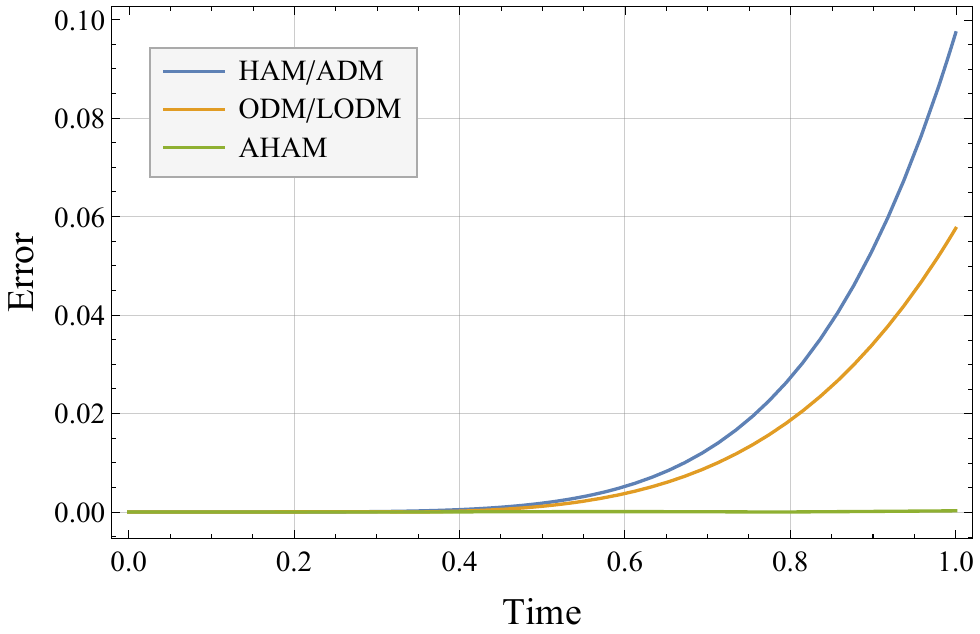}}
	\caption{Number density and error plots for Example \ref{q4}}
	\label{fig10}
\end{figure}
In Figure \ref{fig10}, the depiction illustrates the particle number density at time instance $\tau=1$. It becomes apparent that HAM and ODM exhibit reduced precision over an extended period, whereas AHAM demonstrates a substantial alignment with the precise  number density solution.
 The error due to AHAM is not only significantly smaller than the existing approximated solutions of HAM and ODM but also approaches zero, see Figure \ref{fig10}(b). 
\begin{figure}[htb!]
	\centering
	\subfigure[Zeroth moments]{\includegraphics[width=0.40\textwidth]{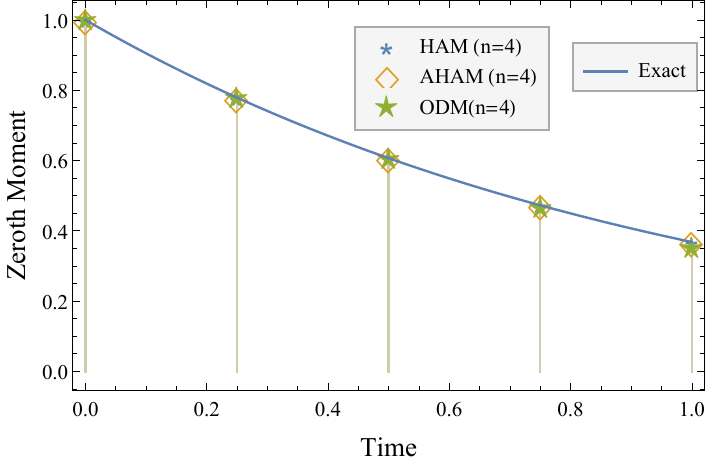}}
	\subfigure[Second moments ]{\includegraphics[width=0.40\textwidth]{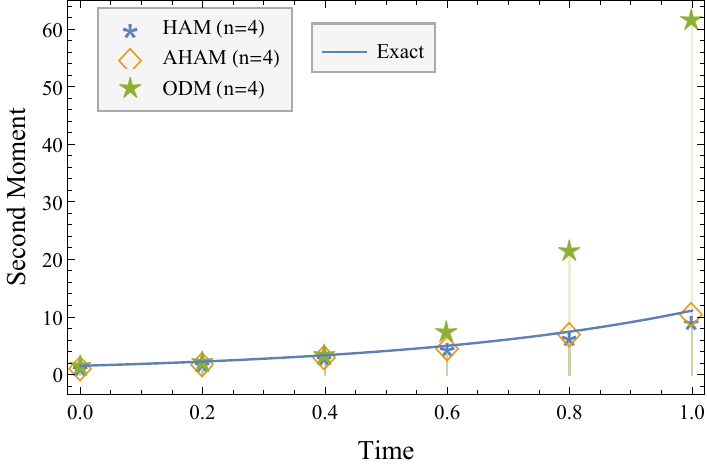}}
	\caption{Zeroth and second moments for Example \ref{q4}}
	\label{fig11}
\end{figure}
Moving further, approximated and exact moments are compared for all three methods in Figure \ref{fig11}. Remarkably, ODM under and over predicts the zeroth and second moments, respectively. Whereas the outcomes from HAM and AHAM exhibit a remarkable congruence, presenting almost identical results. These schemes offer an extremely accurate approximation of the exact zeroth and second moments.
	\begin{example}\label{q5}
We now examine a complicated coagulation kernel that for Brownian coagulation in the free-molecular regime which is dependent on size and was obtained from gas kinetic theory \cite{lee1984log}. Here, the Brownian kernel is transformed into an integrable form as $$w(s,\xi)=b_{k}(s^{1/3}+\xi^{1/3})^{2}\frac{1}{\sqrt{s}+\sqrt{\xi}}$$
where, $k$ denotes the $k^{th}$ moment equation.
Taking the Gaussian initial distribution given in equation \eqref{key} with $\nu$ = 1 and correction factor
$b_{k}$ = 0.7071, the AHAM solution terms are obtained as
	\begin{align*}
		\mu_{0} =& 4se^{-2s} ,\quad
		\mu_{1} =h \tau  e^{-2 s} \bigg( 5.97274 s^{4/3}+3.54487 s^{5/3}+5.34574 s^{5/6}+2.8284 s^{7/6}+2.7273 s+2.68082 \sqrt{s}\\&-7.09296 s^{19/6}\bigg).
		\end{align*}
\end{example}
\begin{figure}[htb!]
	\centering
	\subfigure[Initial distribution and truncated solution for
	number density ]{\includegraphics[width=0.45\textwidth]{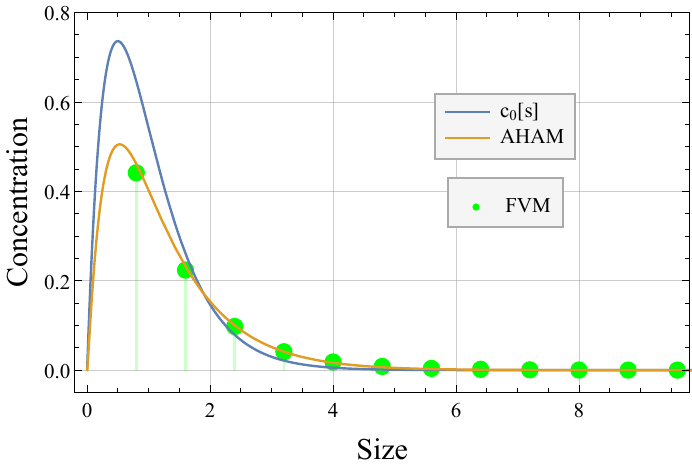}}
	\subfigure[Approximated solution (n = 3) ]{\includegraphics[width=0.45\textwidth]{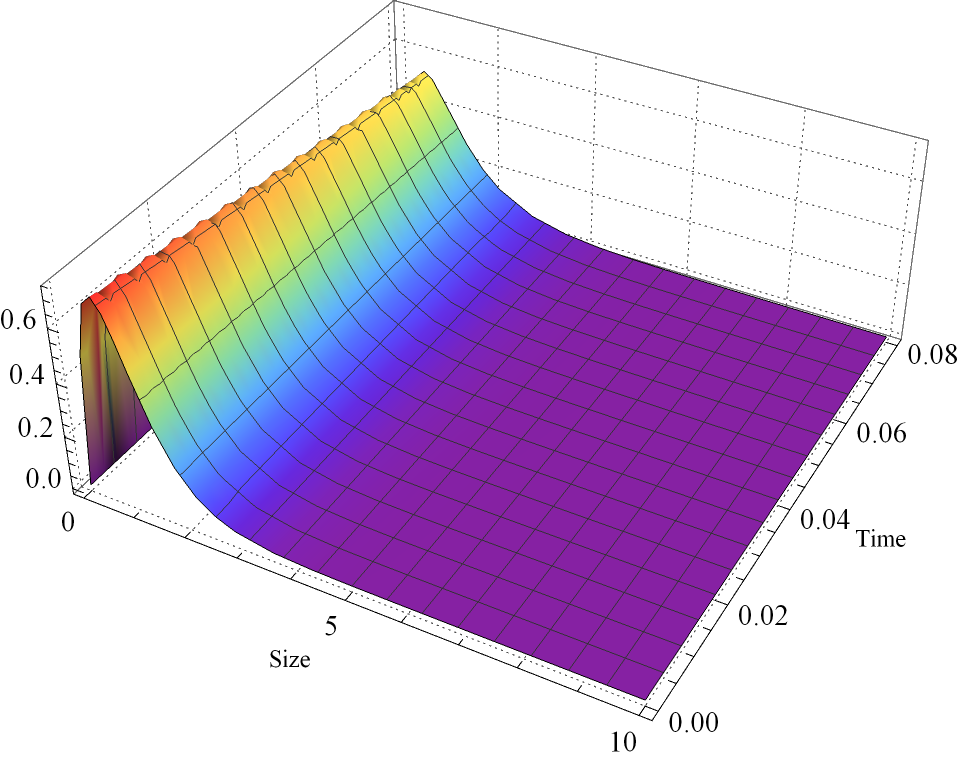}}
	\caption{Approximated solution for Example \ref{q5}}
	\label{fig16}
\end{figure}
The analytical solution for this particular problem is not available in the literature. Consequently, the approximated number density and initial distribution functions are graphically represented alongside the solution derived from the finite volume method, as illustrated in Figure  \ref{fig16} .
The graph indicates a high level of concordance between the AHAM approximation and the Finite Volume Method, demonstrating the anticipated behavior of the model.
\par Based on the illustrations provided above, it is evident that AHAM outperforms ADM, HPM, HAM, ODM and LODM in all contexts. Hence, considering the novelty of the proposed scheme, we intend to use AHAM to solve more intricate equation coupled aggregation-breakage model.
\subsection{Coupled Aggregation-Breakage Equation (CABE)}
\begin{example}\label{q6}
	Consider the CABE \eqref{agg-brk}, with constant aggregation kernel  $w(s,\xi)=1$, binary breakage kernel $\beta(s, \xi) = \frac{2}{\xi}$ and linear selection rate $S(s) = s/2$ with initial condition $c_{0}(s) = 4se^{-2s} $. The exact solution to this problem is available in \cite{lage2002comments}.
	This is the case in which the overall particle count remains constant.
	Using the iterative scheme, first few iterations are obtained as follows
		\begin{align*}
		\mu_{0} =& 4se^{-2s} ,\quad
		\mu_{1} =h \tau e^{-2 s}\left(s^2(\frac{-4s}{3}+2)+ 2 s-1\right) ,\\
		\mu_{2} =&\frac{h \tau e^{-2 s}}{3780} \bigg(2 h^2 \tau^2 s (2 s (s (s (s (21-2 (s-7) s)-210)+105)+315)-315)+63 h \bigg(\tau s \bigg(2 (s-5) s \left(4 s^2-15\right)\\&+15\bigg)-30 \tau+40 s (s (3-2 s)+3)-60\bigg)-1260 (2 s (s (2 s-3)-3)+3)\bigg).
	\end{align*}
\end{example}
Thanks to "MATHEMATICA", one can compute the higher order terms using equation \eqref{iter3}. 
\begin{figure}[htb!]
	\centering
	\subfigure[Error at $\tau$= 1.5 ]{\includegraphics[width=0.40\textwidth]{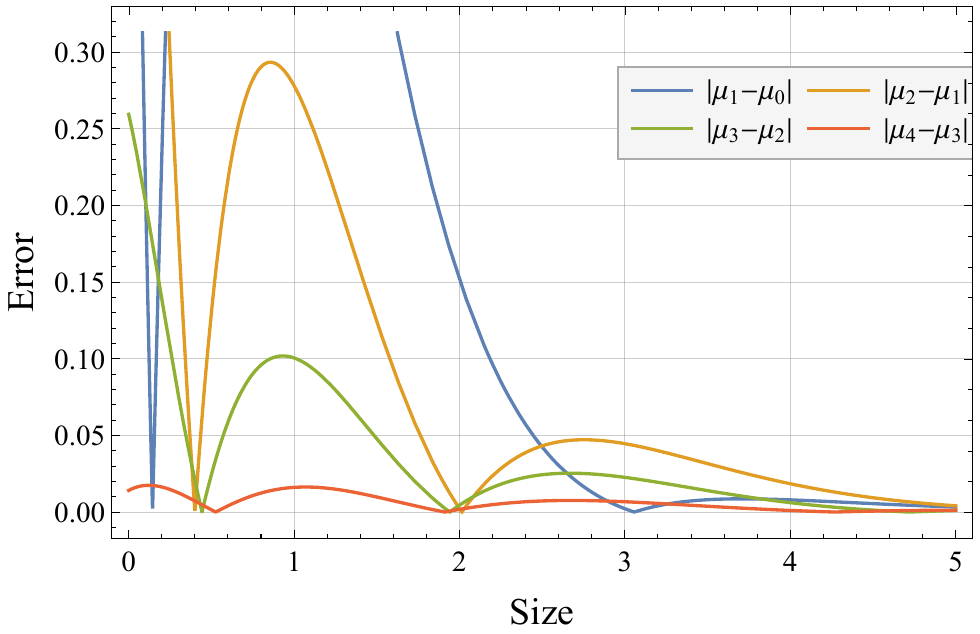}}
	\subfigure[Number density  $(n=4)$]{\includegraphics[width=0.40\textwidth]{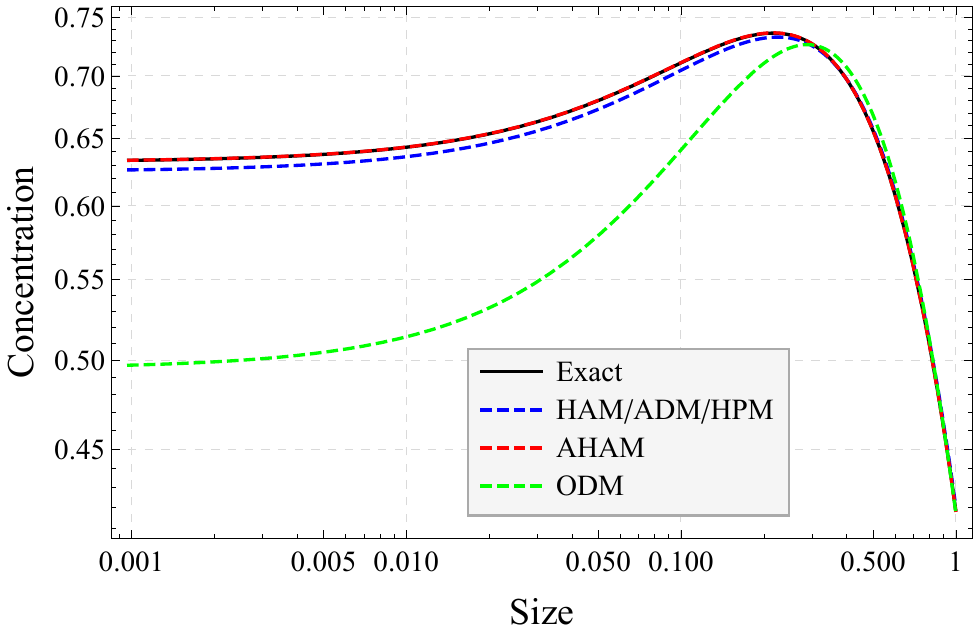}}
	\subfigure[Time distribution ]{\includegraphics[width=0.40\textwidth]{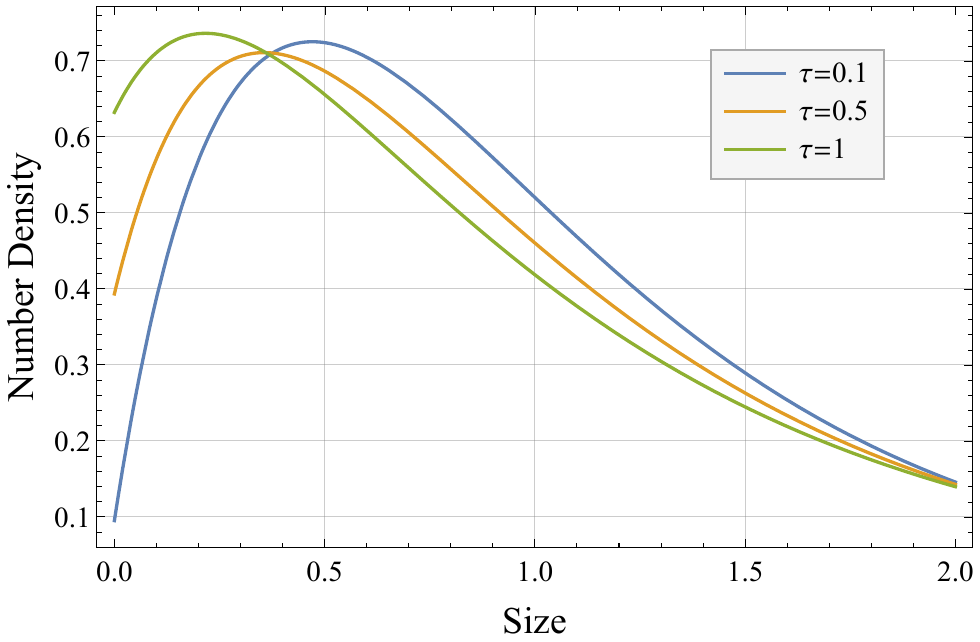}}
	\caption{Number density for Example \ref{q6}}
	\label{fig12}
\end{figure}
Figure \ref{fig12}(a) gives the absolute difference between the subsequent components of the iterative solution. As shown in figure, the consecutive difference among the terms diminishes consistently, and the difference between the third and fourth components is almost inconsequential. This supports truncating the iterative solution at the fourth term. Figure \ref{fig12}(b) illustrates the concentration of particles across various schemes with the exact solution. The figure shows that the AHAM solution precisely aligns with the exact one as compared to HAM and ODM. Figure \ref{fig12}(c) highlights the number density in the system at various time scales. As time passes, smaller particles increase in population while larger particles decrease due to  higher rate of fragmentation over aggregation.
\begin{figure}[htb!]
	\centering
		\subfigure[Truncated error ]{\includegraphics[width=0.40\textwidth]{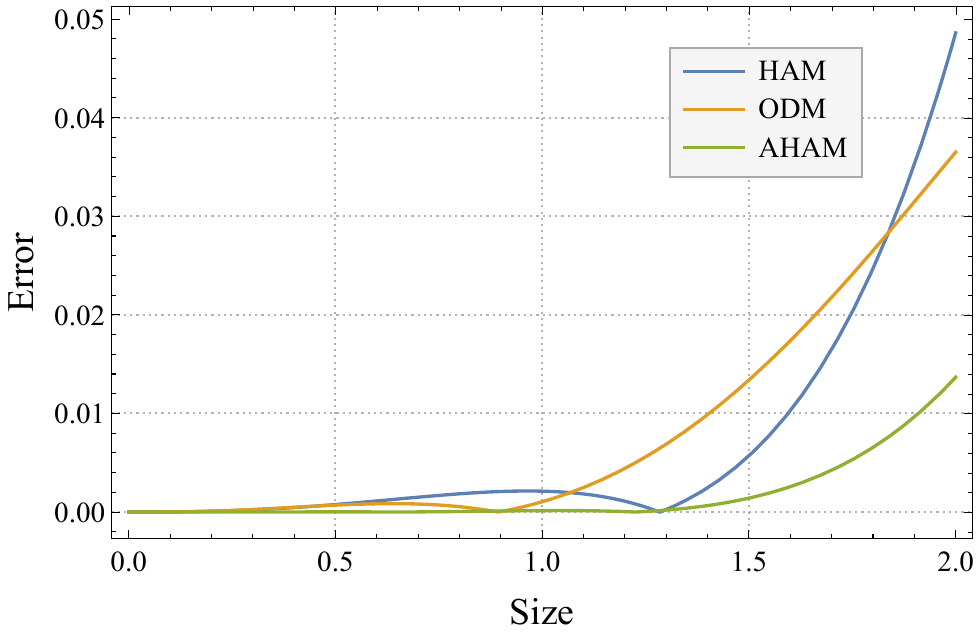}}
		\subfigure[Zeroth moment ]{\includegraphics[width=0.40\textwidth]{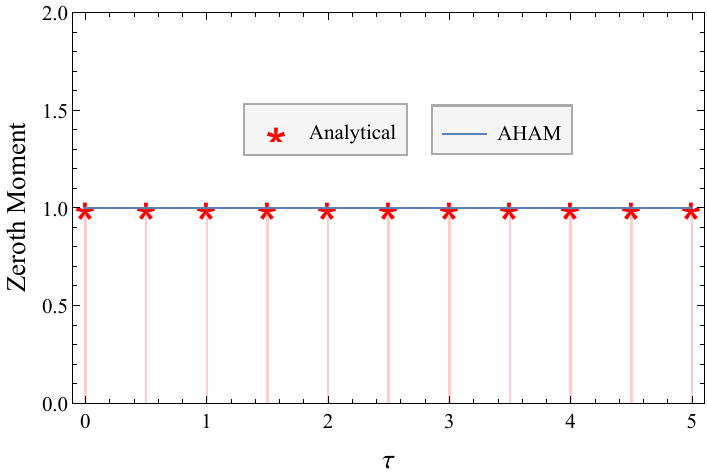}}
	\caption{Error and moment for Example \ref{q6}}
	\label{fig13}
\end{figure}
 Consequently, Figure \ref{fig13}(a) presents the four-term error, demonstrating that the error between the approximated and analytical solutions is negligible, even over a reasonable duration.
The steady-state behavior of the truncated solution concerning the number of particles is illustrated in Figure \ref{fig13}(b), where the zeroth moment remains constant.\\
    In addition, Table \ref{table2} includes numerical errors associated with the AHAM at discrete time instances ($\tau$ = 0.5, 1.0, 1.5, and 2) for different term solutions. To calculate these errors, partition the computional range [0, 10] is partitioned into $K$ subintervals $[s_{j-1/2}, s_{j+1/2}]$, with $j = 1:K$. Each interval is represented by its midpoint $s_{j}$ and the error is computed using the formula:
\begin{align}\label{PARA}
	\text{Error} = \sum_{j=1}^{K}|\mu_n^j-c_j|h_j,
\end{align}  
where $c_j=c(s_j,\tau)$  and $\mu_n^j=\mu_n(s_j,\tau)$ are the exact and series approximated solutions at point $s_j$ with step size $h_j=s_{j+1/2}-s_{j-1/2}$. The results are visualized by taking $K=1000$ and uniform mesh width $h_j=0.01$. 
	\begin{table}
	\begin{center}
			\begin{tabular}{ p{0.5cm}| p{2.5cm} p{2.5cm} p{2.5cm} p{2.5cm}   }
			\hline
			\multirow{2}{*} {$\tau$} &\multicolumn{4}{c}{Number of terms} \\
			\cline{2-5}
			& 2 & 3 & 4 & 5 \\ \hline
			& & & &   \\ 
			0.5 & 1.7944 E-3  & 2.1022 E-4  & 3.13097E-5  & 4.50 E-6    \\ 
			& & & &   \\ 
			1.0 & 2.78838 E-2  & 3.41921 E-3 &  4.32932E-4  & 1.27E-5    \\
			& & & &  \\ 
			1.5 &1.13674E-1  &3.15863 E-2  &6.68227E-3  &8.65E-4    \\
			& & & &   \\ 
			2.0 & 2.68192 E-1 & 1.22675E-1 & 4.17254E-2  & 3.28E-3  \\
			& & & &   \\ 
			\hline
		\end{tabular}
	\end{center}
	\caption{Error distribution at different time level for $n=2$ to $5$ for Example 4.6.}
	\label{table2}
\end{table}
In the realm of Table \ref{table2}, the error gradually increases over time. However, it is possible to reduce this error by incorporating more approximations into the series solution. These findings demonstrate that AHAM is highly consistent and effective, even for CABE. 
\begin{example}\label{q7}
Considering the CABE \eqref{agg-brk} with the same parameters as in Example \ref{q6} but with selection rate $S(s) = 2s$ and the initial condition $c_{0}(s) = 32se^{-4s}$. Again, the steady state precise solution to this problem is available in \cite{lage2002comments}, wherein the zeroth moment remains invariant.
Using the iterations defined in equation \eqref{iter3}, we have
	\begin{align*}
		\mu_{0} =& 32se^{-4s} ,\quad
		\mu_{1} =h \tau e^{-4 s}\left(32s(-\frac{8s^2}{3}+2s+1) -8\right) ,\\
		\mu_{2} =&\frac{8}{945} h \tau e^{-4 s} \bigg(-4 h^2 \tau^2 s (4 s (2 s (4 s (s (4 s (2 s-7)-21)+105)-105)-315)+315)+63 h \bigg(\tau s \bigg(4 s (2 s-5) \bigg(16 s^2\\&-15\bigg)+15\bigg)-15 \tau+20 s \left(-8 s^2+6 s+3\right)-15\bigg)-315 \left(4 s \left(8 s^2-6 s-3\right)+3\right)\bigg).
	\end{align*}
\end{example}
Continuing in a similar fashion, a third order iterative solution is computed and considered for the numerical simulations.
\begin{figure}[htb!]
	\centering
	\subfigure[Number density $(n=3)$]{\includegraphics[width=0.40\textwidth]{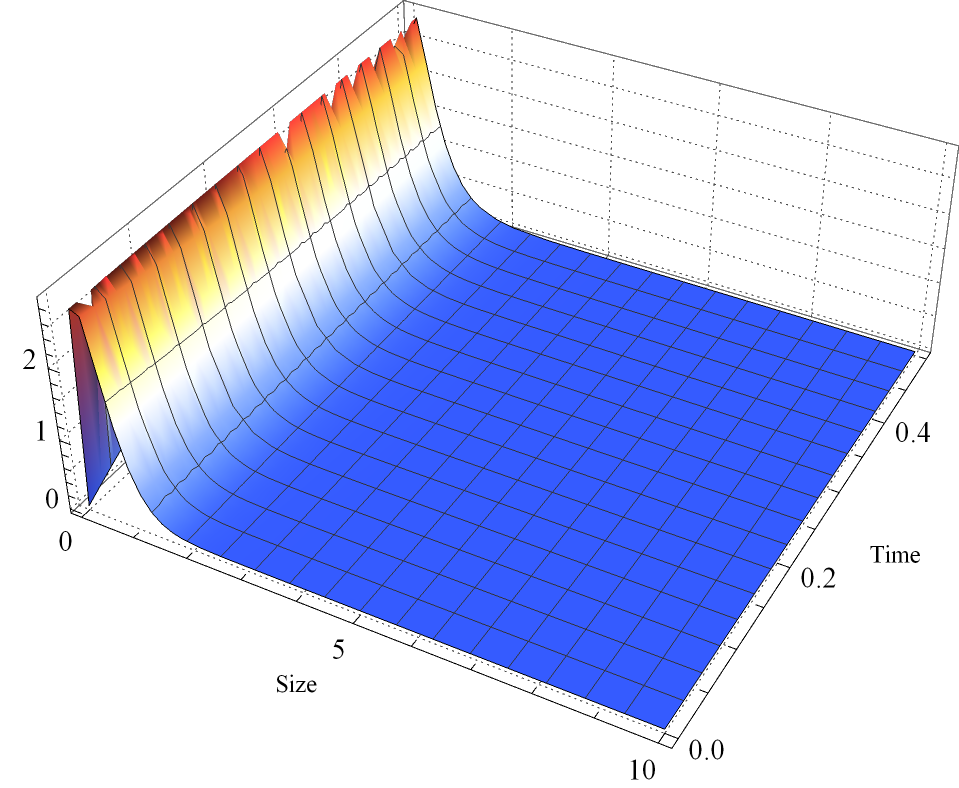}}
	\subfigure[Time distribution ]{\includegraphics[width=0.40\textwidth]{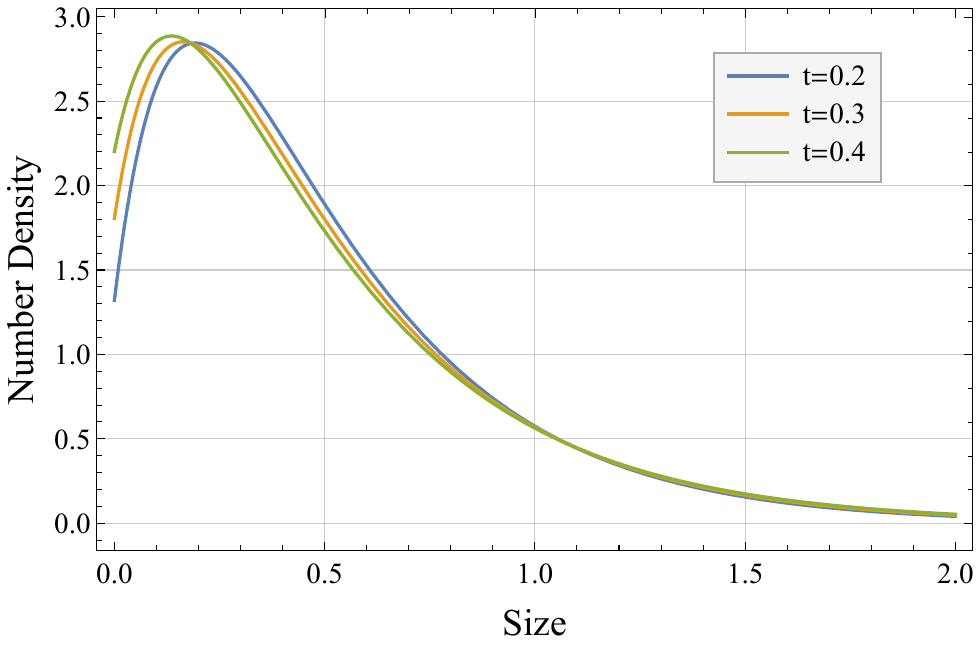}}
	\caption{Number density plots for Example \ref{q7}}
	\label{fig14}
\end{figure}

 Figure \ref{fig14} illustrates the density distribution of particles in the system at various time levels. As time progresses, an increase in the number of smaller particles is observed, while larger particles begin to fragment.
 The difference between consecutive terms is illustrated in Figures \ref{fig15}(a) and \ref{fig15}(b). Remarkably, the error begins diminishing between the second and third terms, prompting us to abbreviate the solution for three terms.
 As anticipated, in Figure \ref{fig15}(c), AHAM exhibits the steady-state character of the zeroth moment and is exactly matching with the exact total number of particles.
\begin{figure}[htb!]
	\centering
	\subfigure[Error at $\tau$ = 0.5 ]{\includegraphics[width=0.40\textwidth]{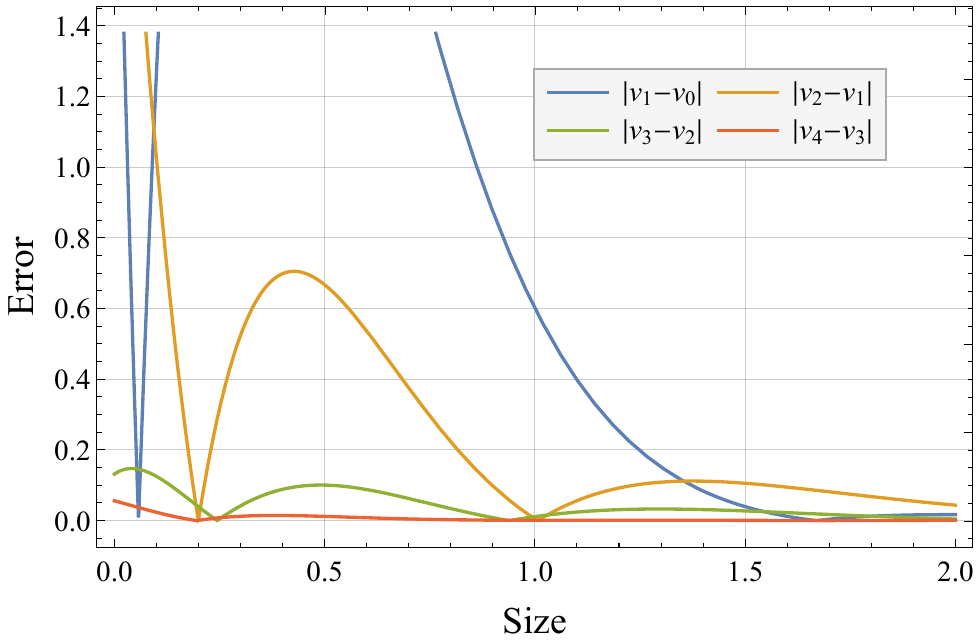}}
	\subfigure[Absolute error ]{\includegraphics[width=0.40\textwidth]{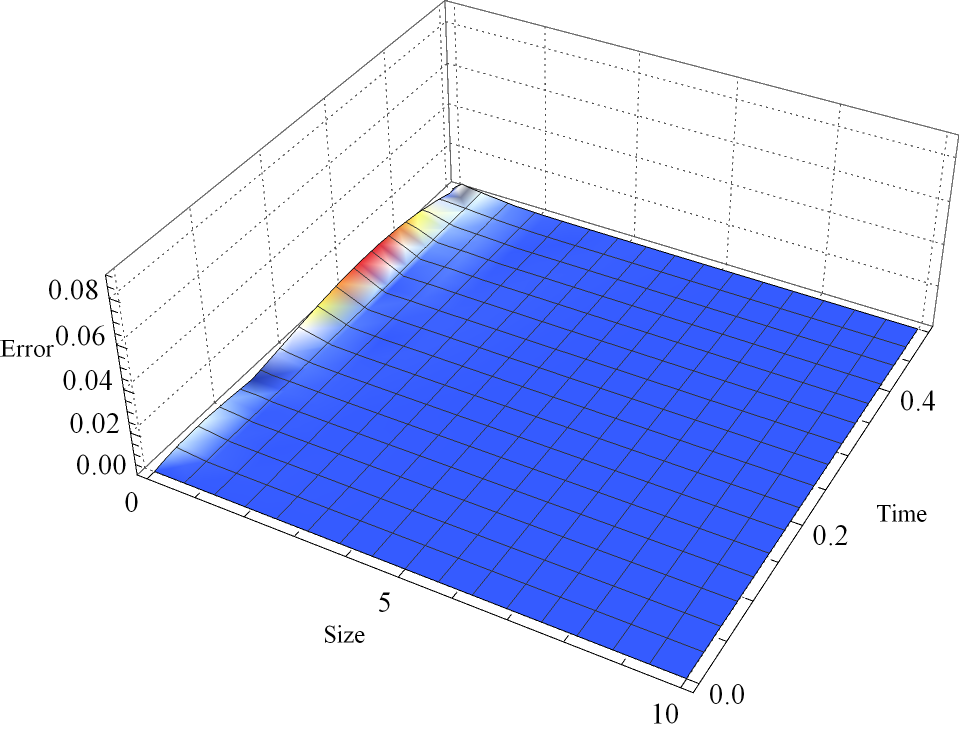}}
		\subfigure[Zeroth moment ]{\includegraphics[width=0.40\textwidth]{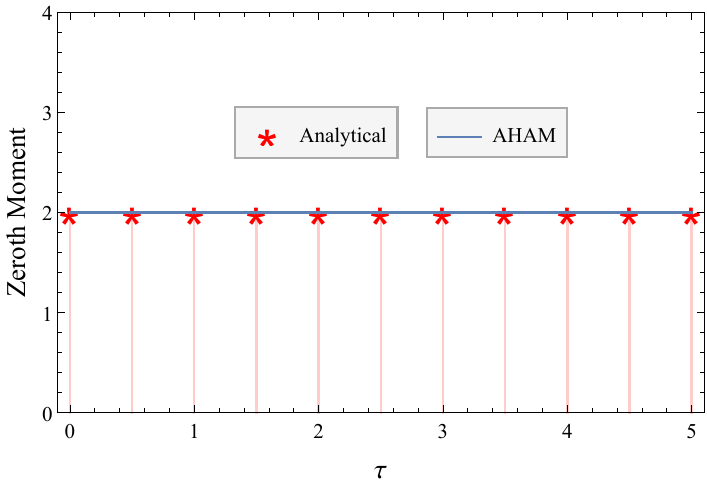}}
	\caption{Error and moment for Example \ref{q7}}
	\label{fig15}
\end{figure}

				\section{Concluding Remarks}
	This study aimed to introduce AHAM, a novel semi-analytical technique rooted in HAM, designed to enhance method efficiency. A comprehensive comparison of AHAM with existing semi-analytical methods, including ADM, ODM, LODM, HPM, and HAM, was conducted, addressing the aggregation breakage equation with various coagulation kernels. Correlations between approximated and exact number density functions and moments were visually and quantitatively presented in figures and tables. Notably, AHAM exhibited superior reliability compared to existing techniques, significantly enhancing HAM's efficiency. The series solutions, accurately replicating analytical solutions where available, demonstrated the anticipated behavior of the model.\\	
	The article also underscored the theoretical convergence analysis of AHAM for both pure aggregation and combined aggregation breakage equations. The comprehensive observations lead to the conclusion that AHAM performed exceptionally well for these specific aggregation breakage equations, suggesting its applicability to studying the behavior of other population balance models with integro-partial differential characteristics.		
					\bibliography{article}
					\bibliographystyle{ieeetr}
				\end{document}